\documentclass[leqno,11pt,a4paper]{amsart}
\usepackage[utf8]{inputenc}

\usepackage{amssymb}
\usepackage{amsfonts}
\usepackage{enumitem}
\usepackage[all]{xy}
\usepackage{mathrsfs}
\usepackage{graphicx}
\usepackage{stmaryrd}
\usepackage{cite}

\usepackage[margin=3cm]{geometry}
\usepackage{mathtools}
\usepackage{arydshln}
\usepackage{multirow}
\usepackage{tikz}
\usepackage{subfigure}
\usepackage{float}
\usepackage{hyperref}
\hypersetup{
hidelinks = true,
}
\usepackage{mleftright}

\graphicspath{{images/}}

\newtheorem{thm}{Theorem}[section]
\newtheorem{prop}[thm]{Proposition}
\newtheorem{lem}[thm]{Lemma}

\newtheorem{conj}[thm]{Conjecture} 

\theoremstyle{definition}
\newtheorem{definition}[thm]{Definition}

\newtheorem{remark}[thm]{Remark}

\newtheorem{construction}[thm]{Construction}

\numberwithin{equation}{section}
\numberwithin{figure}{section}

\DeclareMathOperator{\cl}{cl}

\DeclareMathOperator{\Hom}{Hom}

\DeclareMathOperator{\Ima}{Im}

\DeclareMathOperator{\Region}{\mathbf{T}}

\DeclareMathOperator{\Int}{Int}

\DeclarePairedDelimiter\abs{\lvert}{\rvert}
\newcommand{\mut}{\operatorname{mut}}

\DeclareMathOperator{\GL}{{GL}}
\DeclareMathOperator{\SL}{{SL}}

\DeclareMathOperator{\Ann}{Ann}

\DeclareMathOperator{\Cone}{cone}

\DeclareMathOperator{\Diag}{D}

\newcommand{\cB}{\mathcal{B}}
\newcommand{\ba}{{\mathbf{a}}}

\newcommand{\cT}{\mathbf{T}}

\newcommand{\Q}{{\mathbb{Q}}}
\newcommand{\R}{{\mathbb{R}}}
\renewcommand{\P}{{\mathbb{P}}}
\newcommand{\PP}{{\mathbb{P}}}

\newcommand{\Z}{{\mathbb{Z}}}

\newcommand{\C}{\mathbb{C}}

\newcommand{\Chambers}{\operatorname{Chambers}}
\newcommand{\Rays}{\operatorname{Rays}}

\renewcommand{\emptyset}{\varnothing}
\newcommand{\conv}[1]{\operatorname{conv}\left\{{#1}\right\}}
\newcommand{\V}[1]{\operatorname{\mathcal{V}}\left({#1}\right)}

\newcommand{\fD}{\mathfrak{D}}
\newcommand{\fu}{\mathfrak{u}}

\renewcommand{\d}{\mathfrak{d}}
\newcommand{\fr}{\mathfrak{r}}
\newcommand{\fe}{\mathfrak{e}}
\newcommand{\field}{\mathrm{k}}
\newcommand{\iu}{\imath}
\newcommand{\sS}{\mathscr{S}}
\newcommand{\sP}{\mathscr{P}}
\newcommand{\Markov}{\mathcal{G}}

\begin{document}
\author[T.~Prince]{Thomas Prince}
\address{Mathematical Institute\\University of Oxford\\Woodstock Road\\Oxford\\OX2 6GG\\UK}
\email{thomas.prince@magd.ox.ac.uk}

\title[Tropical Superpotential]{The Tropical Superpotential For $\P^2$}
\begin{abstract}
We present an extended worked example of the computation of the tropical superpotential considered by Carl--Pumperla--Siebert. In particular we consider an affine manifold associated to the complement of a non-singular genus one plane curve, and calculate the wall and chamber decomposition determined by the Gross--Siebert algorithm. Using the results of Carl--Pumperla--Siebert we determine the tropical superpotential, via broken line counts, in every chamber of this decomposition. The superpotential defines a Laurent polynomial in every chamber, which we demonstrate to be identical to the Laurent polynomials predicted by Coates--Corti--Galkin--Golyshev--Kaspzryk to be mirror to $\P^2$.
\end{abstract}

\maketitle

\section{Introduction}
\label{sec:introduction}

The phenomenon of mirror symmetry famously identifies pairs of dual Calabi--Yau manifolds which are related by a duality of $N=2$ superconformal sigma models with Calabi--Yau target spaces. The phenomenon of mirror symmetry extends to some non-Calabi--Yau cases, in particular to the case of \emph{Fano manifolds}~\cite{Givental:Homological,Givental:Toda,Givental:toric,Hori--Vafa:00}. In this setting mirror symmetry is expected to relate a pair $(X,D)$ -- consisting of a Fano manifold $X$ and a divisor $D \in |-K_X|$ -- with a \emph{Landau--Ginzburg model} $(\breve{X},W)$, consisting of a complex manifold $\breve{X}$ and a holomorphic function $W$, referred to as the \emph{superpotential}.

The prototypical example of mirror symmetry for Fano manifolds is the case $(\P^2,D)$ where $D$ is the toric boundary of $\P^2$. The mirror Landau--Ginzburg model in this case is the pair
\[
\big({\C^\star}^2,{\textstyle x+y+\frac{1}{xy}}\big).
\]
Many mathematical formulations of mirror symmetry can be proved in this setting; these include Homological Mirror Symmetry~\cite{Seidel:00,Seidel:01,AKO1}, the Strominger--Yau--Zaslow (SYZ) conjecture~\cite{ChoOh,CL10}, as well as the original formulation of Givental and Hori--Vafa~\cite{Givental:toric,Hori--Vafa:00}. However, even in the case of $(\P^2,E)$ where $E$ is a non-singular genus one curve -- the subject of this article -- the SYZ formulation of mirror symmetry is highly non-trivial.

In general, given a Fano manifold $X$, the SYZ Conjecture~\cite{SYZ} predicts -- see for example the surveys~\cite{Aur1,Aur2} by Auroux -- that the variety $\breve{X}$ is a moduli space of pairs $(L,\nabla)$, where $L$ is a special Lagrangian torus in $X$ and $\nabla$ is a $U(1)$ connection on $L$. Moreover, following \cite{M98}, the moduli space of special Lagrangians on $X$ is a manifold which carries a pair of \emph{integral affine structures}; see Definition~\ref{def:affine_manifold}. In this context the holomorphic function $W$ is predicted to be a count of Maslov index two holomorphic discs in $X$ whose boundary in contained in $L$. 

\begin{figure}
	\includegraphics{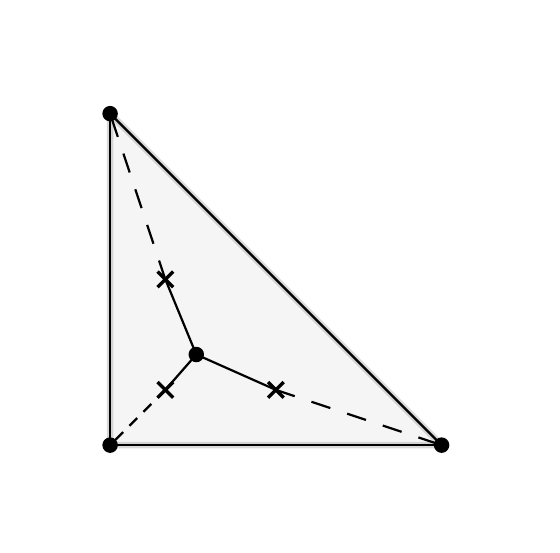}
	\caption{Intersection complex for a toric degeneration of $(\mathbb{P}^2,E)$}
	\label{fig:intersection_complex}
\end{figure}

Fundamental work of Kontsevich--Soibelman~\cite{KS06} and Gross--Siebert~\cite{GS1} exploits this connection between Mirror Symmetry and affine geometry: the authors directly construct a degeneration of a mirror variety via a combinatorial construction (scattering) on an affine manifold. Returning to the case of $\P^2$, one can construct an affine manifold $B$ by `smoothing the corners' of the moment polytope as shown in Figure~\ref{fig:intersection_complex}. This operation appears in \cite{CPS11,Ruddat:LG}, and is explored in some detail in \cite{P15}. There is also an analogue of the superpotential $W$ -- the \emph{tropical superpotential} -- introduced in~\cite{CPS11} in which holomorphic disc counting is replaced by counting tropical discs or \emph{broken lines} in the Legendre dual affine manifold $\breve{B}_{\P^2}$. This builds on the general correspondence between tropical and holomorphic curves established by Mikhalkin~\cite{Mikhalkin03,Mikhalkin05}, and Nishinou--Siebert~\cite{Nishinou--Siebert}.

One essential feature of the Gross--Siebert algorithm is that it encodes the data used to define a degeneration in terms of walls -- or rays -- of a certain wall and chamber decomposition of an affine manifold, created by an order-by-order scattering process. In this article we determine the wall and chamber structure on $\breve{B}_{\P^2}$ produced by the Gross--Siebert algorithm and -- using results of \cite{CPS11} -- determine the tropical superpotential in every chamber. In particular we study a certain collection of rays called a \emph{compatible structure}~\cite[Definition~$6.27$]{TropGeom} $\sS$ on $\breve{B}_{\P^2}$. We recall~\cite[Definition~$6.22$]{TropGeom} that, while the set $\sS$ is infinite, it has a filtration by finite subsets $\sS[k]$, for $k \in \Z_{\geq 0}$. The wall and chamber structure determined by $\sS$ is recorded in a (non-unique) collection of polyhedral subdivisions $\sP_k$ of $\breve{B}_{\P^2}$ for $k \in \Z_{\geq 0}$. The precise conditions the decompositions $\sP_k$ are required to satisfy are given in \cite[Definition~$6.24$]{TropGeom}. Following \cite[p.$276$]{TropGeom}, we let $\Chambers(\sS,k)$ denote the maximal cells of $\sP_k$ for each $k \in \Z_{\geq 0}$.

\begin{thm}
	\label{thm:superpotential}
	There is an increasing sequence of subsets $(V_k)^\infty_{k=0}$ of $\breve{B}_{\P^2}$ and choice of $\sP_k$ for each $k \in \Z_{\geq 0}$ such that, setting $V = \bigcup_{i \geq 0}{V_i}$:
	\begin{enumerate}
		\item For all $j \in \Z_{\geq k}$, the restriction of $\sP_j$ to $V_k$ is a constant union of chambers $\fu \in \Chambers(\sS,k)$.
		\item For each $\fu \in \Chambers(\sS,k)$ such that $\fu \subset V_k$, $\fu$ is related by a scale and translation in an affine chart to a Fano polytope $P_\fu$ whose spanning fan determines a toric variety to which $\P^2$ degenerates.
		\item\label{it:dense} Identifying each ray in $\sS$ with its support in $\breve{B}_{\P^2}$, rays in $\sS$ are dense in $\breve{B}_{\P^2}\setminus V$.
	\end{enumerate}
\end{thm}

We recall that a Fano polygon is a lattice polygon with primitive vertices which contains the origin in its interior. Given a Fano polygon $P$, its \emph{spanning fan} is the rational fan obtained by taking cones over the faces of $P$. We refer to~\cite{TropGeom,GS1} for the precise definitions of the terms specific to the Gross--Siebert algorithm, although we provide an overview of the program in \S\ref{sec:GS_background}. Note that item~\eqref{it:dense} of Theorem~\ref{thm:superpotential} is proven in \S\ref{sec:proof} subject to Conjecture~\ref{conj:dense_with_rays}, a well known expectation on rank $2$ scattering diagrams. Combining Theorem~\ref{thm:superpotential} with the results of \cite{CPS11} we recover the tropical superpotential in every chamber.

\begin{thm}
	\label{thm:superpotential_2}
	The \emph{tropical superpotential} $W^k_{\omega,\tau,\fu}$ is \emph{manifestly algebraic} in the sense defined in~\cite[p.$34$]{CPS11}, and may be identified with the unique rigid maximally mutable Laurent polynomial~\cite[Definition~$4$]{Overarching} with Newton polygon $P_{\fu}$.
\end{thm}

That is, rather than a single mirror Laurent Polynomial $W$, we obtain infinitely many polynomials related by certain birational changes of variables. The dual cell complex to these walls and chambers is a trivalent tree, and each chamber is a triangle similar to the Fano triangle defined by the corresponding degeneration of the projective plane. In fact the nodes of this trivalent tree have an interpretation as integral solutions of the Markov equation $x^2+y^2+z^2=3xyz$. Indeed, we recall that toric degenerations of $\P^2$ are also classified by such Markov triples; see Hacking--Prokhorov~\cite{HP10}.

\begin{thm}
	\label{thm:P2_classification}
	The set of toric varieties to which $\P^2$ admits a toric degeneration is in canonical bijection with the integral solutions of the \emph{Markov equation} $a^2+b^2+c^2 = 3abc$. Consequently all toric degenerations of $\P^2$ are related by combinatorial mutation.
\end{thm}

We recall that combinatorial mutation was defined in~\cite{ACGK12} for any Fano polytope. The Markov equation is well known and appears in many different areas of mathematics. The integral solutions of this equation are completely described by the following lemma.

\begin{lem}
	Given a solution $(a,b,c)$ of $x^2 + y^2 + z^2=3xyz$ another solution is given by $\left(b,c,3bc-a\right)$. Given the initial solution $(1,1,1)$ this process generates all integral solutions to the Markov equation. 
\end{lem}

Thus the solutions of the Markov equation may be encoded in a trivalent graph $\Markov$, part of which is illustrated below. Note that -- fixing the root $(1,1,1)$ -- the tree $\Markov$ can be interpreted as a partial order of its vertices. Regarded as a partially ordered set, $\Markov$ is an unbounded meet-semilattice. $\Markov$ is graded by the function $d\colon \Markov \to \Z_{\geq 0}$, such that -- for any $\ba \in \Markov$ -- $d(\ba)$ is the length of the shortest path between $\ba$ and $(1,1,1)$. Given an element $\ba \in \Markov$, we let $\Markov_{\mathbf{a}}$ denote the subgraph of $\Markov$ on vertices greater than or equal to $\mathbf{a}$.

\begin{center}
	\small
	\begin{tikzpicture}[grow=down,level distance=1cm]
	\tikzstyle{level 3}=[sibling distance=6cm] 
	\tikzstyle{level 4}=[sibling distance=3cm] 
	\tikzstyle{level 5}=[sibling distance=1.5cm,level distance=0.7cm]
	\node {$(1,1,1)$}
	child {node {$(1,1,2)$}
		child {node {$(1,2,5)$}
			child {node {$(2,5,29)$}
				child {node {$(5,29,433)$}
					child {node {} edge from parent[dotted]}
					child {node {} edge from parent[dotted]}}
				child {node {$(2,29,169)$}
					child {node {} edge from parent[dotted]}
					child {node {} edge from parent[dotted]}}}
			child {node {$(1,5,13)$}
				child {node {$(5,13,194)$}
					child {node {} edge from parent[dotted]}
					child {node {} edge from parent[dotted]}}
				child {node {$(1,13,34)$}
					child {node {} edge from parent[dotted]}
					child {node {} edge from parent[dotted]}}}}
	};
	\end{tikzpicture}
\end{center}

In fact this structure on the mirror to the projective plane is expected from another point of view on Mirror Symmetry. Following \cite{CCGGK}, one expects that a certain period of the fibration defined by $W$ computes a certain generating function of Gromov--Witten invariants called the (regularised) \emph{quantum period}. In the case of $\P^2$ it is easy to construct an infinite family of Laurent polynomials $f \in \C[\Z^2]$ such that the period integral
\[
\pi_f(t) := \int_\Gamma \frac{\Omega}{1-tf},
\]
\noindent where $\Gamma := \{|x_1| = |x_2| = 1\}$ and $\Omega := \frac{1}{2\pi\iu}\frac{dx_1}{x_1}\wedge\frac{dx_2}{x_2}$, is equal to the regularised quantum period $\widehat{G}_{\P^2}$ of $\P^2$,
\[
\widehat{G}_{\P^2}(t) = \sum_{m \geq 0}{\frac{(3m)!}{(m!)^3}t^{3m}}.
\]
Indeed, following \cite{CCGGK} we say that a Laurent polynomial $f \in \C[\Z^2]$ is \emph{mirror-dual} to $\P^2$ if its period $\pi_f(t)$ is equal to the regularised quantum period of $\P^2$. A collection of such polynomials, indexed by the vertices of $\mathcal{G}$, can be obtained from the polynomial $x+y+1/xy$ using the notion of \emph{mutation of potential}\footnote{These are called \emph{algebraic mutations} in \cite{ACGK12}, and \emph{symplectomorphisms of cluster type} in~\cite{Katzarkov--Przyjalkowski}.} defined by Galkin--Usnich~\cite{GU10}, and developed by Akhtar--Coates--Galkin--Kasprzyk~\cite{ACGK12}.

\begin{remark}
	Combining results of Tveiten~\cite{T18} with the results of \cite{Overarching}, it may be shown that \emph{all} Laurent polynomials with period $\pi_f(t) = \widehat{G}_{\P^2}$ may be obtained from the polynomial $x+y+1/xy$ by mutation.
\end{remark}

Theorem~\ref{thm:superpotential_2} shows that all Laurent polynomials mirror-dual to $\P^2$ can be expressed as counts of broken lines in the affine manifold $\breve{B}_{\P^2}$. We also observe that the integral solutions to the Markov equation also enumerate the monotone Lagrangian tori in $\P^2$ found by Vianna~\cite{Vianna:P2}. In fact Theorems~\ref{thm:superpotential} and~\ref{thm:superpotential_2} can be interpreted as tropical analogues of these results in symplectic geometry. Indeed, for each chamber in $\breve{B}_{\P^2}$ the SYZ conjecture predicts there is a family of Hamiltonian isotopic Lagrangian tori lying in $\P^2$, thus we demonstrate the close compatibility of the wall and chamber structure defined by the Gross--Siebert algorithm and that predicted by the existence of non-Hamiltonian isotopic Lagrangian tori in this case.

\section*{Acknowledgements}
We thank Alexander Kasprzyk and Mohammad Akhtar for explaining combinatorial mutations and thank Tom Coates for suggesting a number of corrections. There is also a clear intellectual debt owed to the preliminary version of the paper~\cite{CPS11} of Carl--Pumperla--Siebert. TP is supported by an EPSRC Postdoctoral Prize Fellowship and Fellowship by Examination at Magdalen College, Oxford.

\section{Fano Polygons and mutation}
\label{sec:fano_polygons}

We study the class of Fano polygons $P$ associated\footnote{We say a polygon $P$ is associated to a toric variety $X$ if $X$ is isomorphic to the toric variety defined by the \emph{spanning fan} of $P$.} to $\Q$-Gorenstein toric degenerations of $\P^2$. These polygons are related by \emph{combinatorial mutation}; we refer to~\cite{ACGK12} for the general definition of a combinatorial mutation, but we give a simple characterisation of the definition in Lemma~\ref{lem:P2_mutations}.

Given a Fano polygon $P$ we can form a pair $(n,\cB)$ -- the \emph{singularity content}~\cite[Definition~$3.1$]{SingCon} of $P$ -- consisting of an integer $n$ and a set of cyclic quotient singularities $\cB$. The combinatorial definition of singularity content and proof of its mutation invariance is given in~\cite{SingCon}, but a geometric interpretation of this invariant is given in~\cite[p.$10$]{Overarching}. In particular, given a Fano polygon $P$ and a locally $\Q$-Gorenstein rigid del Pezzo surface $X$ which admits a $\Q$-Gorenstein degeneration to $X_P$, the integer $n$ is the topological Euler characteristic of the smooth locus of $X$. The tuple $\cB$ is the basket of singularities of the log del Pezzo surface $X$.

The class of Fano polygons associated to $\Q$-Gorenstein toric degenerations of $\P^2$ is well-understood; see Theorem~\ref{thm:P2_classification}, following \cite{HP10}. Indeed, each Fano polygon in this class is formed by taking the convex hull of the ray generators of the fan determined by a weighted projective space $\P(a^2,b^2,c^2)$, where $(a,b,c)$ is a Markov triple. Note that, via the characterisation of singularity content given above, the singularity content of any such Fano triangle $P$ is equal to $(3,\varnothing)$. Fix a lattice $N \cong \Z^2$, and let $N_\R := N\otimes_\Z \R$.

\begin{thm}
	\label{thm:P2_triangles}
	Given a Fano polygon $P \subset N_\R$ the following are equivalent:
	\begin{enumerate}
		\item\label{it:toric_degen} $P$ is the polygon associated to a toric degeneration of $\P^2$.
		\item\label{it:sing_con} The singularity content of $P$ is $(3,\emptyset)$.
		\item\label{it:mutations} The polygons $P$ and $P_0 := \conv{(1,0),(0,1),(-1,-1)}$ are related by a sequence of combinatorial mutations.
	\end{enumerate}
\end{thm}
\begin{proof}
	The equivalence of \eqref{it:sing_con} and \eqref{it:mutations} follows from the mutation invariance of singularity content~\cite[Proposition~$3.6$]{SingCon} and the classification of minimal polygons~\cite[Theorem~$1.2$]{KNP15}. The equivalence of \eqref{it:toric_degen} and \eqref{it:mutations} follows from the classification of Hacking--Prokhorov stated in Theorem~\ref{thm:P2_classification}.
\end{proof}

\begin{figure}
	\includegraphics[scale=0.7]{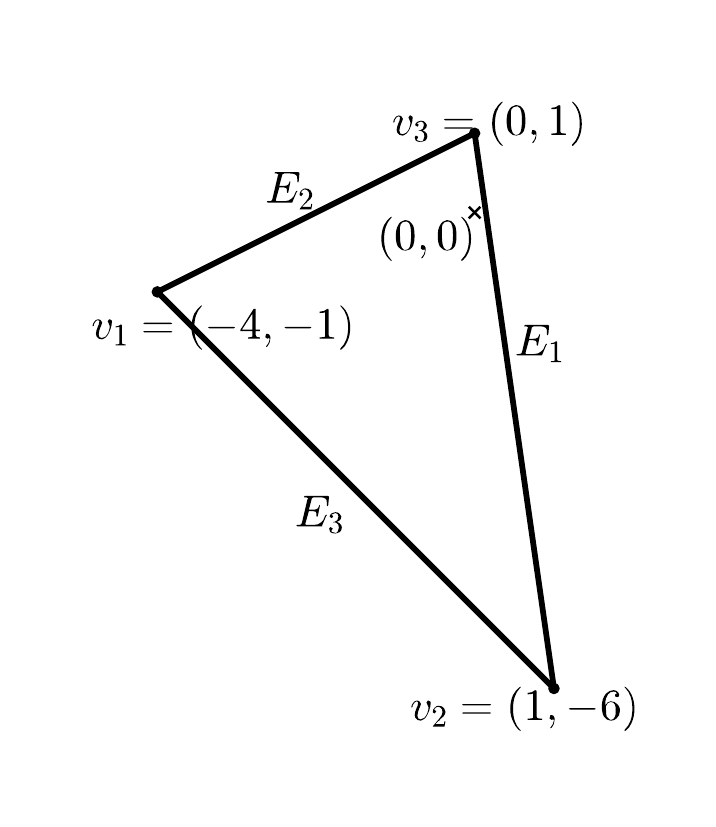}
	\caption{Fano polygon associated to $\P(1^2,2^2,5^2)$ }
	\label{fig:example_triangle}
\end{figure}

For the remainder of this section we fix a Markov triple $\ba = (a_1,a_2,a_3)$, and let $P \subset N_\R$ be the Fano polygon associated to $\P(a_1^2,a_2^2,a_3^2)$. We also fix a bijection between the vertices of $P$ and the multiset $\{a_1,a_2,a_3\}$ such that, if $v_i$ denotes the vertex associated to $a_i$ for $i \in \{1,2,3\}$, $\sum_{i = 1}^3{a_i^2v_i} = 0$. Let $E_i$ denote the edge of $P$ which is disjoint from $v_i$ for each $i \in \{1,2,3\}$. Throughout this article we will assume that Markov triples are ordered so that $a_3 \geq a_2$ and $a_3 \geq a_1$.

Fixing an edge $E$ of $P$, let $w_E \in M := \Hom(N,\Z)$ denote the primitive inner normal vector to $E$. The integer $r_E = -w_E(E)$ is called the \emph{local index} of $E$. It is easy to verify that $r_{E_i} = a_i$, and that the lattice length $\ell(E_i) = a_i$ for all $i \in \{1,2,3\}$.

Following \cite[Definition~$5$]{ACGK12}, a \emph{mutation} of $P$ is fixed by a choice of weight vector $w \in M$ and factor polytope $F \subset \Ann(w) \otimes_\Z \R$. In fact, as $\Ann(w) \cong \Z$, we can make a standard choice of $F$; the interval $\conv{0,u}$, where $u$ is a primitive lattice vector. Note that there is a binary choice of $u$, which we leave unresolved. There are three (non-trivial) choices of weight vectors $w$ for $P$: the inner normal vectors to the edges. We let $\mut_{w}(P)$ denote either of the ($\SL_2(\Z)$ equivalent) polygons obtained by mutating with either choice of the element $u$.

Fix a weight vector $w \in M$ and factor $F = \conv{0,u} \subset w^\bot$, where $u$ is a primitive lattice vector, which define a mutation of $P$. Note there is a unique edge $E$ and vertex $v$ of $P$ such that $E = \ell(E) F + v$. Letting $v'$ denote the vertex of $P$ not contained in $E$, the following lemma may be taken as the definition of the polygon $\mut_w(P)$. See Figure~\ref{fig:mutating_triangle} for an illustration of this transformation.

\begin{lem}
	\label{lem:P2_mutations}
	The polygon $\mut_w(P)$ is equal to the convex hull of $v$, $v'$, and $v' + \langle w, v'\rangle u$.
\end{lem}

\begin{figure}[H]
	\centering
	\subfigure{%
		\includegraphics[scale=0.8]{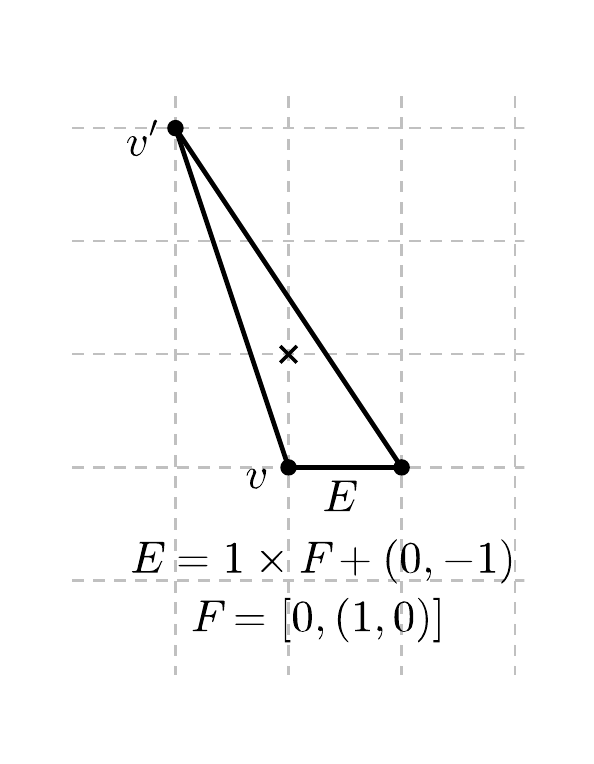}}
	\quad
	\subfigure{%
		\includegraphics[scale=0.8]{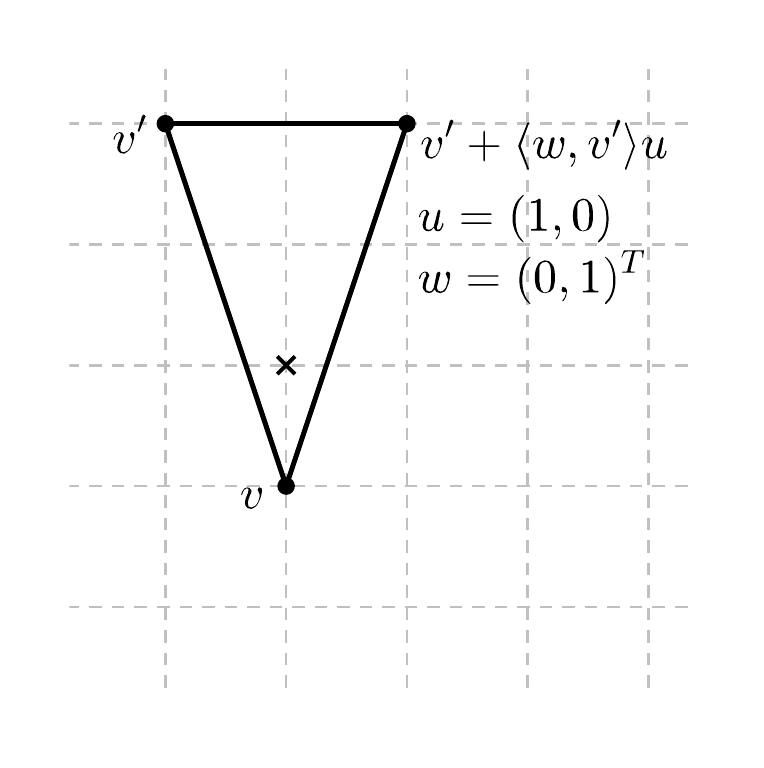}}

	\caption{Mutating a Fano triangle. \label{fig:mutating_triangle}}
\end{figure}

\begin{remark}
	The polygon $\mut_w(P)$ is exactly -- that is, not only up to transformations in $\SL(2,\Z)$ -- determined by a \emph{mutating} edge $E$ and \emph{fixed} edge $E'$. In Figure~\ref{fig:mutating_triangle} the edge $E'$ is the convex hull of $v$ and $v'$.
\end{remark}

\begin{remark}
	\label{rem:alg_mutation}
	Combinatorial mutation is the operation on polytopes induced by taking the Newton polytopes of Laurent polynomials related by algebraic mutations~\cite{ACGK12}. In the two-dimensional case these are precisely the birational transformations
	\[
	\theta_{w,F}^\star(z^n) = F^{\langle w,n \rangle}z^n,
	\]
	where $F = (1+z^u)$ and $u \in w^\bot$ is a primitive lattice point.
\end{remark}

The effect of a mutation on the triple of inner normals to the edges of $P$ is also easy to describe. Recall $E_i$ denotes the edge of $P$ disjoint from $v_i$ for each $i \in \{1,2,3\}$, and let $w_i$ be the primitive inner normal vector to $E_i$. We have the following formula for the inner normal vectors of the polygon obtained by mutating $P$ along $E_i$,

\begin{equation}
\label{eq:PL_map}
w_j \mapsto
\begin{cases}
-w_i & \text{if $i=j$} \\
w_j + \max\{0,w_i\wedge w_j\} w_i & \text{otherwise.}
\end{cases}
\end{equation}

This transformation is illustrated in Figure~\ref{fig:changing_normals}. Note that there is a binary choice in this formula; the choice of the orientation of $M$ used to identify $\bigwedge^2 M$ with $\Z$.

\begin{figure}
	\makebox[\textwidth][c]{\includegraphics[width=0.8\textwidth]{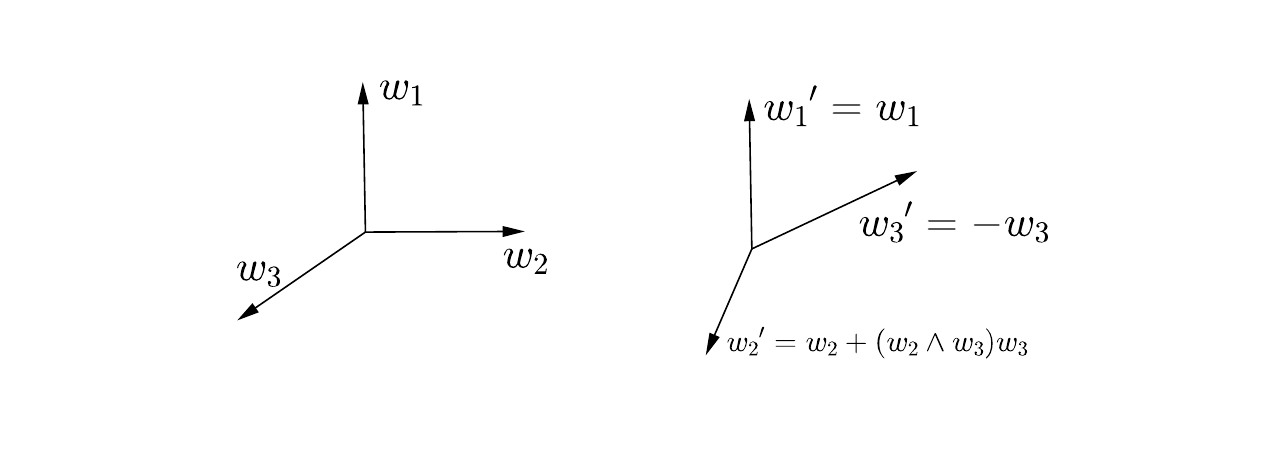}}%
	\caption{Normal vectors changing under a mutation}
	\label{fig:changing_normals}
\end{figure}

We describe a `normal form' for the Fano $P$ by mapping $P$ into $\R^2$ and making the edges of $P$ incident to the vertex $v_3$ orthogonal, at the expense of embedding $N$ into a finer lattice.

	
\begin{definition}
	\label{dfn:std_form}
	Consider the map $\rho\colon \Z^2 \rightarrow M$ defined by sending $\rho \colon e_i \mapsto w_i$ for each $i \in \{1,2\}$, where $\{e_i : i \in \{1,2\}\}$ is the standard basis of $\Z^2$. The dual map $\rho^\star\colon N \hookrightarrow \Z^2$ embeds $N$ into the lattice $(\Z^2)^\star$, and consequently embeds $P$ into the vector space $(\R^2)^\star$. We refer to the embedding $\rho^\star$ as the \emph{normal form} for $P$.
\end{definition}
Figure~\ref{fig:normal_form_ex} shows the Fano polygon associated to $\P^2$ in standard form.
\begin{figure}
	\centering
	\includegraphics[scale=1.2]{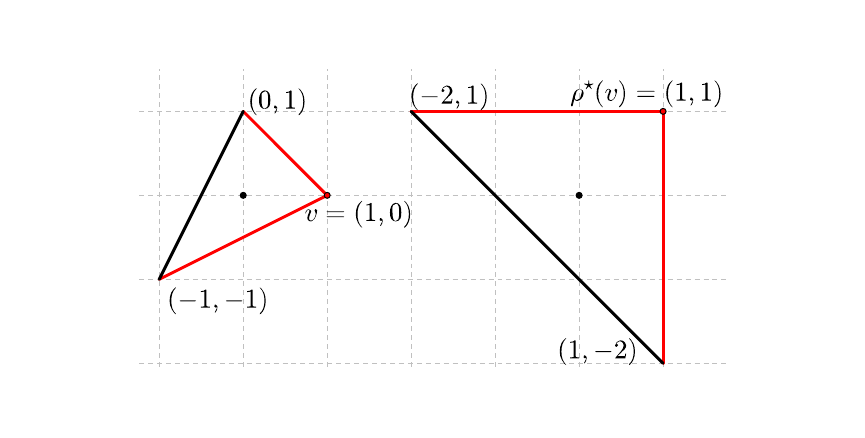}
	\caption{A triangle in normal form}
	\label{fig:normal_form_ex}
\end{figure}
\begin{remark}
	Definition~\ref{dfn:std_form} is intimately related to the construction of a \emph{cluster algebra} associated to $P$ described in \cite{KNP15}, c.f.~\cite{GHK2}. In that context one defines \emph{seed data} by fixing the lattice $\Z^m$ with its standard basis, and a skew-symmetric bilinear form $\{-,-\}$ on $\Z^m$. The mutation of seed data then involves making a choice of basis vector $e_k$ and applying the transformation
	\[ e_i' = 
	\begin{cases}
	-e_k & \text{if } i=k \\
	e_i + \max{(\{e_k,e_i\},0)}e_k & \text{otherwise}.
	\end{cases}
	\]
	Given a Fano polygon $Q$ with singularity content $m_Q = \sum_E{m_E}$, where $m_E$ is the singularity content\footnote{See~\cite[Definition~$2.4$]{SingCon} for the definition of the singularity content of a cone.} of $\Cone(E)$, and the sum is taken over the edges of $Q$, we can define a map $\hat{\rho} \colon \Z^{m_Q} \rightarrow M$ sending $m_E$ basis vectors to the inward-pointing normal vector to the edge $E$. We define a rank $2$ skew-symmetric $2$-form on $\Z^m$ by setting $\{u_1,u_2\} := \det(\hat{\rho}(u_1),\hat{\rho}(u_2))$. Restricting this definition to a pair of basis vectors we recover our previous definition of $\rho$.
\end{remark}

Putting $P$ in normal form embeds the edges $E_1$ and $E_2$ of $P$ into affine coordinate lines. This means we have a very simple description of the result of the pair of mutations in the edges $E_1$ and $E_2$. To make this precise we define the notion of polygon mutation \emph{with respect to a sublattice}.

\begin{definition}
	Given an inclusion $\rho^\star\colon N \hookrightarrow \Z^2$, a Fano polygon $Q \subset N_\R$, a vector $w \in (\Z^2)^\star$, and a line segment $F = \rho^\star(F')$ -- where $F' \subset w^\bot$ -- we define the \emph{mutation with respect to $N$} as follows,
	\[
	\mut_{w,\rho^\star}(\rho^\star(Q),F') := \rho^\star(\mut_{\rho(w)}(Q,F)).
	\]
\end{definition}

\begin{lem}
	\label{lem:facet_length}
	Put the Fano polygon $P$ in normal form. The mutations of $\rho^\star(P)$ with respect to $N$ in edges $\rho^\star(E_1)$ and $\rho^\star(E_2)$ have factors
	\begin{align*}
	F_1 = \conv{(0,0),(\pm s,0)} && \textrm{and} && F_2 = \conv{(0,0),(0,\pm s)}
	\end{align*}
	\noindent respectively, where $s$ is the index $[\Z^2 : N]$.
\end{lem}
\begin{proof}
	The factors $F_i$ are, by definition, line segments such that the origin is a vertex. Since $F_i$ lies in $w_i^\bot$,
	\begin{align*}
	F_1 = \conv{(0,0),(k_1,0)} && \textrm{and} && F_2 = \conv{(0,0),(0,k_2)}
	\end{align*}
	\noindent for some integers $k_1$ and $k_2$. To see that $k_i = \pm s$ for each $i \in \{1,2\}$ we observe that
	\[
	k_1 =  \langle e_2,\rho^\star(\nu_1)\rangle = \nu_1(\rho(e_2)) = \langle w_2, \nu_1\rangle,
	\]
	where $\nu_i$ is a primitive integral vector parallel to $E_i$. Fixing an orientation of $M_\R$, the functional $\langle -,\nu_1\rangle\colon M \rightarrow \Z$ is equal to $w_1\wedge -$. Since $s$ is the determinant of $\rho$, $\abs{w_1\wedge w_2} = s$.
\end{proof}

We fix notation for the local indices of cones over edges in $\rho^\star(P)$ and its mutations in $E_1$ and $E_2$. For each $i \in \{1,2\}$ let $r_i$ be the local index of the cone over the edge $E_i$, and let $r'_i$ be the local index of the cone over the edge $E'_i$, where $E'_i$ is the edge formed by mutating at $E_i$. 
	\begin{itemize}
		\item Let $r_i$ be the local index of the cone over the edge $E_i$.
		\item Let $r'_i$ be the local index of the cone over the edge $E'_i$, where $E'_i$ is the edge formed by mutating at $E_i$. 
		\item Let $R_i = r_i + r'_i$.
	\end{itemize}

\begin{lem}
	\label{lem:standard_scale}
	Defining $R_i$ and $r_i$ as above for $i \in \{1,2\}$, we have that $R_1/r_2 = R_2/r_1 = s$.
\end{lem}
\begin{proof}
	Since $\ell(E_i)$ is equal to the local index of the cone over $E_i$, the mutation in these edges completely removes the edge. Consequently the mutation of $\rho^\star(P)$ with respect to $N$ also removes an entire edge. If $p$ is a point on $E_i$, $r_i = \langle w_i, p \rangle = \langle e_i, \rho^\star(p)\rangle$, thus the local index of the cone over $\rho^\star(E_i)$ equals that of the cone over the edge $E_i$. Hence $R_i = r_{3-i}s$ for $i \in \{1,2\}$ by Lemma~\ref{lem:facet_length}.
\end{proof}
\section{Gluing Fano polygons}
\label{sec:diagrams}

We now consider a construction of the support of a \emph{scattering diagram}\footnote{See \S\ref{sec:GS_background} for a special case and~\cite{GPS,TropGeom,GS1} for more details.} in terms of polygon mutations. In particular we build a collection of triangles using successive mutations, and use the edges of these polygons to define a collection of rays. The main objects of study in this section are collections of triangles formed by mutations which we call \emph{diagrams}.

Throughout this section we fix a Fano polygon $P \subset N_\R$ associated to a Markov triple $\ba = (a_1,a_2,a_3)$ and assume, without loss of generality, that $a_1 \leq a_3$, and $a_2 \leq a_3$. For each $i \in \{1,2,3\}$ let $v_i$ and $E_i$ denote the vertex and edge associated with $a_i$.

Let $\Diag_0(P) := \{P\}$ and, for $i \in \{1,2\}$, let $P^i_1$ denote the polygon obtained by mutating $P$ at the edge $E_i$ while fixing $E_{3-i}$. Let $E^i_1$ be the edge of $P^i_1$ equal to $E_{3-i}$, let $F^i_1$ be the edge of $P^i_1$ corresponding to $E_3$ after mutation, and let $G^i_1$ be the remaining edge of $P^i_1$ (corresponding to $E_i$ under mutation). For any $k \in \Z_{> 0}$, let $P^i_{k+1}$ be the polygon obtained by mutating $P^i_k$ at $E^i_k$ and fixing edge $F^i_k$. Let $E^i_{k+1}$ be the edge corresponding to $G^i_k$ after mutation, $G^i_{k+1}$ the edge corresponding to $E^i_k$, and let $F^i_{k+1}$ be the edge of $P^i_{k+1}$ equal to $F^i_k$.

Given a Fano polygon $Q$, let $R$ be a polygon obtained from $Q$ by mutating along edge $E$ of $Q$ which fixes an edge $F$. Moreover let $E'$ be the edge of $R$ corresponding to $E$ after mutation. Let $S_{Q,E,F}$ denote the (unique) map $x \mapsto \lambda x + u$ such that $\lambda \in \R_{>0}$, $u \in N_\R$, and $S_{Q,E,F}(E') = E$. We use these maps to glue together an infinite collection of Fano polygons.

\begin{definition}
	\label{dfn:diagram}
	 Let $T^i_1 := S_{P,E_i,E_{3-i}}$, and set $\fu^i_1 := T^i_1(P^i_1)$ for $i \in \{1,2\}$. For $k \in \Z_{> 0}$ we set $T^i_{k+1} := T^i_k\circ S_{P^i_k,E^i_k,F^i_k}$ and $\fu^i_{k+1} := T^i_{k+1}(P^i_{k+1})$. Let $\Diag_0(P) := P$, and for $k \in \Z_{\geq 0}$ let 
	 \[
	 	\Diag_{k+1}(P) := \Diag_k(P) \cup \{ \fu^1_{k+1}, \fu^2_{k+1} \},
	 \]
	and $\Diag(P):= \bigcup_{k \geq 0}{\Diag_k(P)}$. We also write $\fu^i_j(P)$ for the triangle $\fu^i_j \in \Diag(P)$, where $i \in \{1,2\}$ and $j \in \Z_{\geq 0}$. We refer to any set $D_k(P)$ or $D(P)$ as a \emph{diagram}.
\end{definition}

Given an affine linear map $L \colon \R^2 \to \R^2$ we define $L(\Diag(P)) := \{L(\fu) : \fu \in \Diag(P)\}$.

\begin{remark}
	The condition that $a_3$ is maximal in the triple $(a_1,a_2,a_3)$ means that the vertex $v_3$ of $P$ is uniquely determined unless $a_1 = a_2 = a_3 = 1$. In this case we may label the vertices arbitrarily, and our convention will be to write $\Diag_k(P,v)$ where $v \in \V{P}$ is the vertex chosen to be $v_3$. 
\end{remark}


\begin{figure}
	\makebox[\textwidth][c]{\includegraphics[scale=0.85]{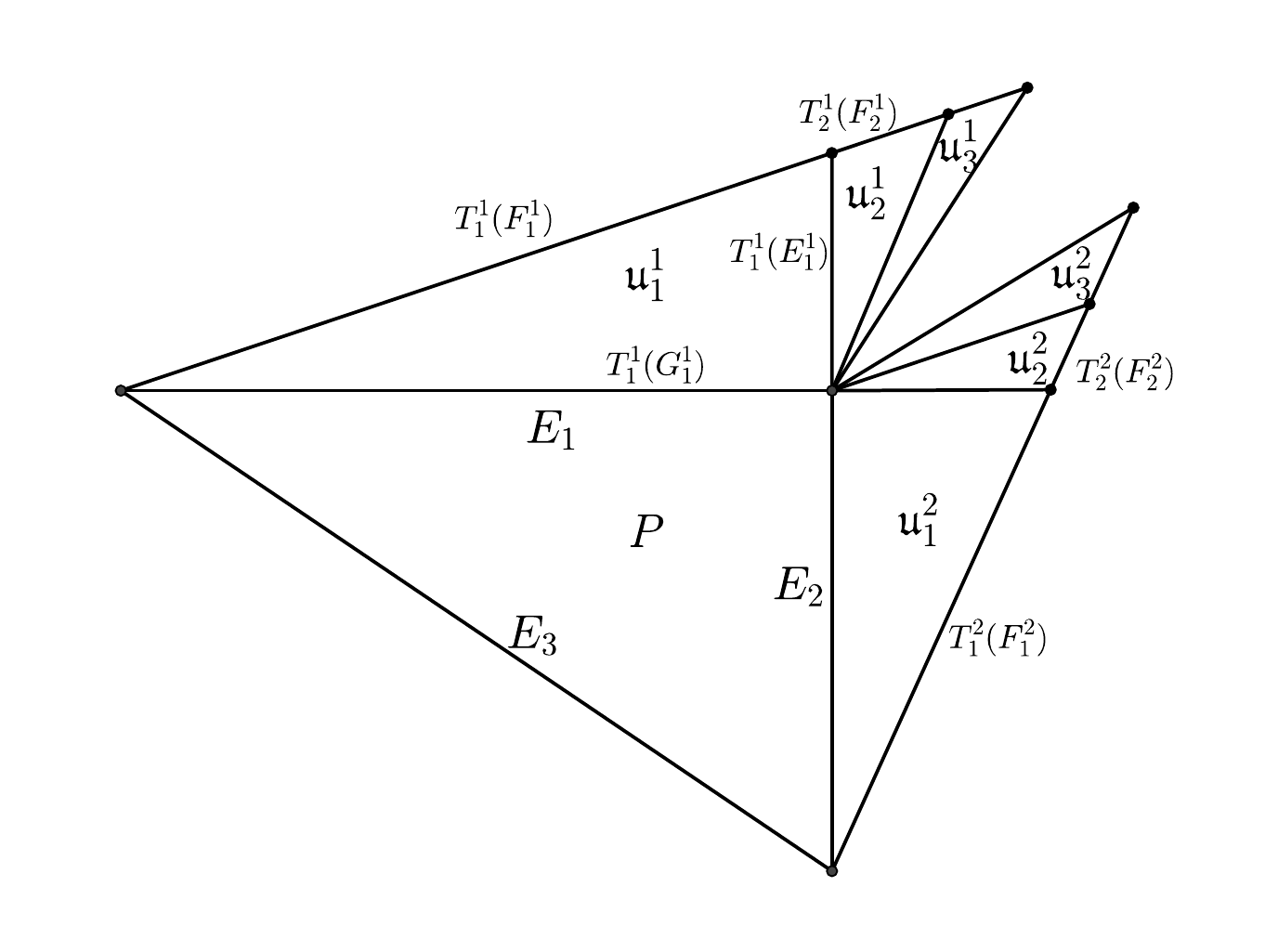}}%
	\caption{Triangles appearing in $\Diag_3(P)$.}
	\label{fig:adding_rays}
\end{figure}

For the remainder of this section we assume $P$ is in standard form with respect to $v$, that is, we embed $P$ into $\Z^2$ via the map $\rho^\star$, where $\rho$ sends the standard basis in $\Z^2$ to the primitive inner normal vectors to $E_1$ and $E_2$; recall that we let $s := \det(\rho)$. We slightly abuse notation and write $\fu$ for $\rho^\star(\fu)$ for each $\fu \in \Diag(P)$ (including $P$) in what follows.

We use transformation \eqref{eq:PL_map} to describe the normal vectors of the region $\fu^i_j$. To do this we define vectors $u^i_{j,k}$ normal to edges of $\fu^i_j$, illustrated in Figure~\ref{fig:normals}.

\begin{definition}
	Let $H_1$ and $H_2$ denote the edges of $\fu^i_j(P)$ incident to $v$. Order this pair of edges so that $H_1$ is an edge of $\fu^i_{j-1}(P)$, and $H_2$ is an edge of $\fu^i_{j+1}(P)$. Let $u^i_{j,1}$,~$u^i_{j,2} \in \R^2$ be the inner normal vectors to $H_1$ and $H_2$ in $\fu^i_j(P)$ respectively.
\end{definition}

The first few terms of the sequences $u^1_{j,1}$, and $u^1_{j,2}$ are shown below. Note these two sequences are equal to the two rows displayed; the arrows indicate how the normal vectors transform under mutation.

\begin{figure}
	\centering
	\begin{minipage}{0.45\textwidth}
		\centering
		\includegraphics[width=\textwidth]{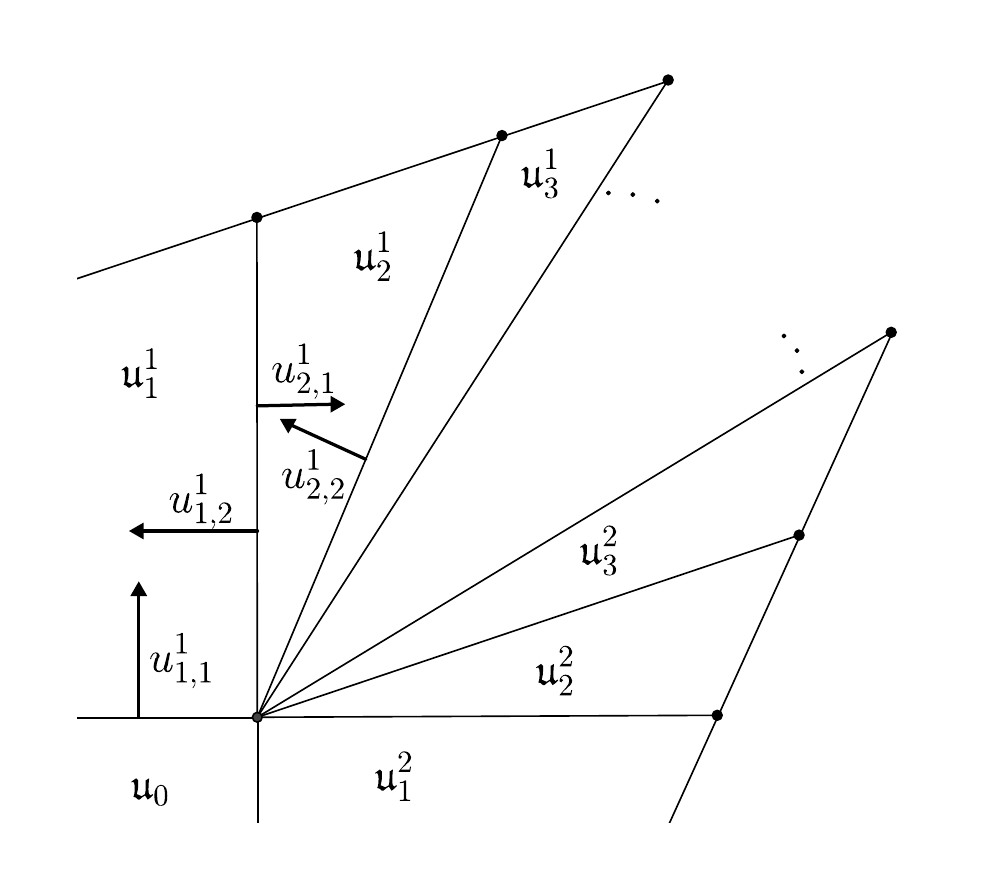}
		\caption{Normal vectors $u^i_{j,k}$.}
		\label{fig:normals}
	\end{minipage}\hfill
	\begin{minipage}{0.45\textwidth}
		\centering
		\includegraphics[width=\textwidth]{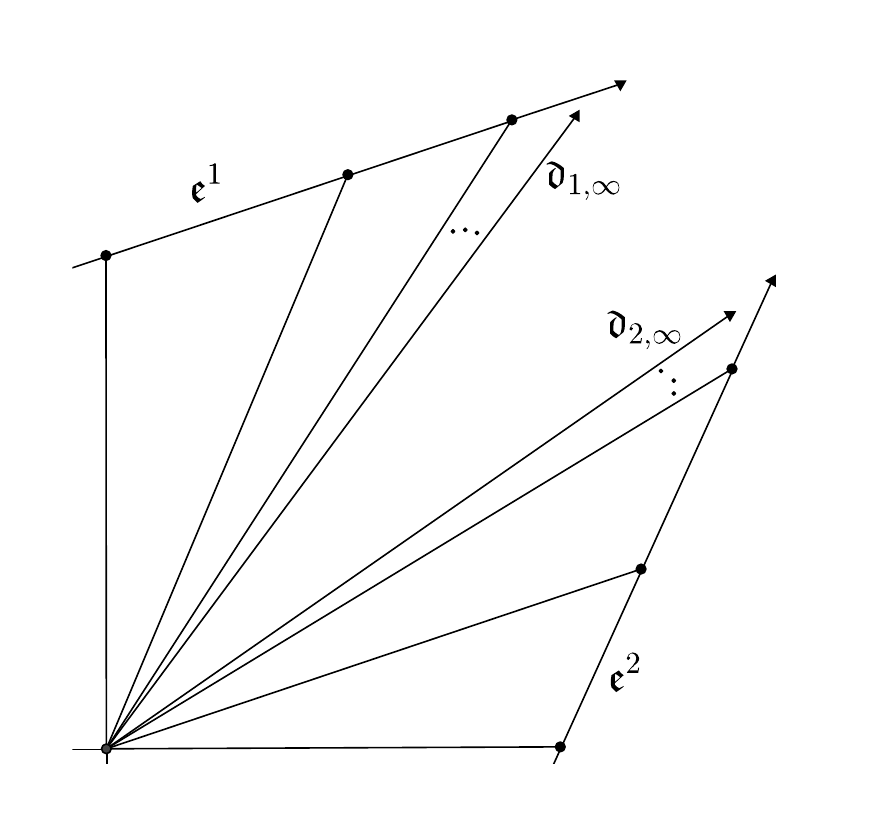}
		\caption{Rays $\fe^i$ and $\d^i_\infty$.}
		\label{fig:slope_bound}
	\end{minipage}
\end{figure}

\begin{align*}
\xymatrix{
	(0,1) \ar[dr] & (1,0) \ar[dr] & (s,-1) \ar[dr] & (s^2-1,-s) \ar[dr] & \cdots \ar[dr] & u^1_{1,j} \\
	(-1,0) \ar[ur] & (-s,1) \ar[ur] & (1-s^2,s) \ar[ur] & (2s-s^3,s^2-1) \ar[ur] & \cdots \ar[ur] & u^2_{1,j}
}
\end{align*}
There is an analogous sequence of vectors $u^2_{j,k}$ for $k \in \{1,2\}$. Recalling that inner normal vectors to edges transform under mutation by the piecewise linear transformation \eqref{eq:PL_map} we compute that
\begin{align*}
u^1_{j,1} = (x_j,-x_{j-1}) && \textrm{and} && u^2_{j,1} = (-x_{j+1},x_j),
\end{align*}

for values $x_j$ generated by the recursive relation
\[
x_{j+1} = sx_j - x_{j-1}
\]
such that $x_0 = -1$, and $x_1 = 0$.

\begin{lem}
	\label{lem:normal_asymptotes}
	For each $i \in \{1,2\}$ let $u^i_{\infty,1} := \lim_{j \to \infty} u^i_{j,1}$. The linear span of the vector $u^i_{\infty,1}$ is a line $L$ in $\R^2$ with slope $\frac{-s+\sqrt{s^2-4}}{2}$. 
\end{lem}
\begin{proof}
	The line spanned by $u^1_{j,1}$ has slope $m_j = -x_j/x_{j+1}$. From the recurrence relation above we have that $m_{j+1}m_j + m_js + 1 = 0$, from which we compute the limiting slope. The line spanned by $u^1_{j,2}$ has slope $m_{j-1}$, and hence the same limiting slope. 
\end{proof}

\begin{remark}
	Repeating the calculation appearing in Lemma~\ref{lem:normal_asymptotes} for the sequences $u^2_{j,1}$ and $u^2_{j,2}$, the sequence of slopes of the corresponding lines converges to $\frac{-s-\sqrt{s^2-4}}{2}$.
\end{remark}

We compute the slopes of the rays in $T_v\R^2$ generated by edges of the chambers $\fu^i_j$.
\begin{definition}
	\label{dfn:rays_from_edges}
	Let $\d^i_j$ be the sequence of rays generated by vectors $v^i_j$ for $i \in \{1,2\}$, $j \in \Z_{\geq 0}$. The vectors $v^i_j$ are defined recursively as follows.
	\begin{itemize}
		\item $v^1_0 := (-1,0)^T$,~$v^1_1 := (0,1)^T$,
		\item $v^2_0 := (0,-1)^T$,~$v^2_1 := (1,0)^T$, and,
		\item $v^i_{j-1} + v^i_{j+1} = s v^i_j$, recalling that $s = \det(\rho)$.
	\end{itemize}
\end{definition}

\begin{lem}
	\label{lem:recursive_rays}
	The set of rays $\{\d^i_j : i \in \{1,2\}, ~j \in \Z_{>0}\}$ in $T_v\R^2$ is equal to the set of rays in $T_v\R^2$ induced by edges of the triangles in $\Diag(P)$ incident to $v$.
\end{lem}
\begin{proof}
	There are two edges of each chamber $\fu^i_j$ incident to $v$. Since, for $j \in \Z_{>0}$, ${\left(u^1_{j,1}\right)}^\bot = (x_{j-1},x_j)$, and $x_{j+1} = sx_j - x_{j-1}$, these vectors satisfy the recurrence relation defining $v^i_j$.
\end{proof}

In particular, fixing $i \in \{1,2\}$, the rays $\d^i_j$ converge to a pair of asymptotes as $j \to \infty$.
\begin{lem}
	\label{lem:asymptote_slope}
	The sequence of rays $\d^i_j$ converge to the ray $\d^i_\infty$ with slope
	\[
	m^i_\infty := \frac{s - (-1)^i\sqrt{s^2 - 4}}{2}.
	\]
\end{lem}
\begin{proof}
	Fixing an $i \in \{1,2\}$, the limit of the vectors $v^i_j$ as $j \to \infty$ lies in the normal space to the limit obtained in Lemma~\ref{lem:normal_asymptotes}.
\end{proof}

Each chamber $\fu^1_j$ has edges with inner normals $u^1_{j,k}$ for $k \in \{1,2\}$, as well as the edge $F_j := T^1_j(F^1_j)$. It follows immediately from the transformation of normal vectors of a triangle under mutation that the sequence of inner normals vectors to the edges $F_j$ is constant. Thus these edges determine a ray --  which we denote $\fe^i$ for $i \in \{1,2\}$ -- illustrated in Figure~\ref{fig:slope_bound}. For each $i \in \{1,2\}$, let $\ell_i := \ell(E_i)$ be the lattice length of $E_i$.

\begin{lem}
	\label{lem:comparison}
	Fix a triangle $P$ associated to a Markov triple $\ba = (a_1,a_2,a_3)$, and consider $P$ in standard form with respect to the vertex $v_3$ corresponding to $a_3$. We have that $\ell_1 = a_1$,~$\ell_2 = a_2$ and $s=3a_3$.
\end{lem}
\begin{proof}
	Fixing an orientation of $N_\R$, let $s_i$ denote the wedge product of the inner normal vectors to the edges of $P$ incident to $v_i$ for each $i \in \{1,2,3\}$. By \cite[Proposition~$3.17$]{KNP15}, mutating $P$ in edge $E_1$ sends the triple $(s_1,s_2,s_3)$ to $(s'_1,s_2,s_3)$, where $s_1' = s_2s_3-s_1$. Noting that 
	\[
	\frac{s'_1}{3} = 3\frac{s_2}{3}\frac{s_3}{3} - \frac{s_1}{3},
	\]
	and that $s_1=s_2=s_3=3$ if $P = \conv{(1,0),(0,1),(-1,-1)}$, we have that $\ba = (s_1/3,s_2/3,s_3/3)$.
	
	To prove $\ell_1 = a_1$ and $\ell_2 = a_2$, first recall that $\ell_2$ is equal to the local index of the edges $E_1$,~$E_2$ respectively. Additionally recalling that the triple of local indices of $P$ is given by the Markov triple $\ba$, the result follows.
\end{proof}

The following bound ensures that once the rays $\fe^i$ enter the region defined by the pair of asymptotes $v^i_\infty$ of the two sequences of rays they never emerge again. The rays $\fe^1$, $\fe^2$, $\d^1_\infty$, and $\d^2_\infty$ are illustrated in Figure~\ref{fig:slope_bound}. 

\begin{lem}
	\label{lem:slope_bound}
	Let $m_i$ be the slope of the rays $\fe^i$ for $i \in \{1,2\}$, then
	\[
	m^1_\infty > m_i > m^2_\infty.
	\]
\end{lem}
\begin{proof}
	We first compute the gradient of the rays $\fe^i$. The edge $E_3$ of $P$ (in normal form) has normal vector $(\ell_1,\ell_2)$. Thus, transforming this normal direction using \eqref{eq:PL_map}, the normal directions to the edges $F^1_1$ and $F^2_1$ of $P^1_1$ and $P^2_1$ respectively are $(\ell_1,\ell_2-s\ell_1)$ and $(\ell_1-s\ell_2,\ell_2)$. Thus the corresponding tangent directions have slopes $\frac{\ell_1}{s\ell_1-\ell_2}$ and $s-\frac{\ell_1}{\ell_2}$. Note that -- since we can freely interchange $a_1$ and $a_2$ -- we only need to consider the second case; that is, we only need to prove that
	\[
	\frac{s-\sqrt{s^2-4}}{2} < s-\frac{\ell_1}{\ell_2} < \frac{s+\sqrt{s^2-4}}{2}.
	\]
	By Lemma~\ref{lem:comparison} these are equivalent to the inequality
	\[
	\sqrt{(3a_3)^2-4} > \abs*{3a_3 - \frac{2a_1}{a_2}}.
	\]
	Squaring both sides and rearranging, we may reduce this inequality to a tautology:
	\begin{align*}
	(3a_3)^2-4 > (3a_3)^2 - 4.3\frac{a_3a_1}{a_2} + 4\frac{a_1^2}{a_2^2} & \Leftrightarrow \\
	\frac{a_1}{a_2}\left(\frac{3a_3a_2-a_1}{a_2}\right) > 1 & \Leftrightarrow \\
	1+\frac{a_3^2}{a_2^2} > 1.
	\end{align*}
\end{proof}

One easy consequence of Lemma~\ref{lem:slope_bound} is that triangles $\fu^1_j$ never overlap triangles $\fu^2_k$ for non-negative integers $j$ and $k$. Phrased differently, we have the following proposition.

\begin{prop}
	The collection of triangles $\Diag(P)$ are the maximal cells of a polyhedral decomposition of a subset of $\R^2$.
\end{prop}

\section{Gluing diagrams}
\label{sec:gluing_diagrams}

We construct polyhedral decompositions of subsets of $\R^2$ which extend those defined in the previous section. The triangular regions in these decompositions will form chambers of the compatible structure $\sS$ which appears in the statement of Theorem~\ref{thm:superpotential}. Throughout this section we assume that $P$ is a Fano triangle associated to a Markov triple $\ba = (a_1,a_2,a_3)$ such that $a_3 \geq a_1$ and $a_3 \geq a_2$. For each $i \in \{1,2,3\}$ let $v_i$ and $E_i$ denote the vertex and edge of $P$ associated to $a_i$, following the prescription given in \S\ref{sec:diagrams}.

Given such a Fano polygon -- and setting $A_0(P) := \{P\}$ -- we define a collection $A_k(P)$ of subsets of $N_\R$ for each $k \in \Z_{\geq 0}$. For any $\fu \in A_k(P)$ we inductively assume that there is a unique Fano polygon $P_\fu$ and affine transformation $S_\fu$ of the form $S_\fu \colon x \mapsto \lambda x + u$, for some $\lambda \in \R_{>0}$ and $u \in N_\R$, such that $S_\fu(P_\fu) = \fu$. Given the collection $A_k(P)$ we define
\[
A_{k+1}(P) := \{S_{\fu}(\fu') : \fu \in A_k(P), \fu' \in \Diag_1(P_\fu) \setminus \{P_\fu\}\}.
\]
Note that for any $\fu \in A_{k+1}(P)$ there is a unique $P_\fu$ and $S_\fu$ such that $S_\fu(P_\fu) = \fu$.

\begin{definition}
	Assuming that $\ba \neq (1,1,1)$, let $\Region(P) := \coprod_{k \geq 0}{A_k(P)}$ be the union of the sets $A_k(P)$ for $k \in \Z_{\geq 0}$. Assuming instead that $\ba = (1,1,1)$, we let $\Region(P,v)$ denote the union of the sets $A_k(P)$, where $\Diag_1(P_\fu)$ is replaced by $\Diag_1(P_\fu,v)$ in the definition of $A_1(P)$.
\end{definition}

Polygon mutation induces a partial order on $\Region(P)$ (or $\Region(P,v)$) and, if $\ba \neq (1,1,1)$, there is an order preserving bijection between $\Region(P)$ and the graph $\Markov_{\ba}$. If $\ba = (1,1,1)$, there is an order preserving injection $\Region(P,v) \hookrightarrow \Markov_\ba = \Markov$, fixed by taking the two mutations of $\ba$ corresponding to the edges incident to the vertex $v$ of $P$. In particular $\Region(P)$ (resp. $\Region(P,v)$) is a graded meet-semilattice and we let $\fu \wedge \fu'$ denote the infimum of $\fu$ and $\fu'$. We say $\fu'$ is a \emph{successor} of $\fu$ if $\fu' > \fu$ and $d(\fu') = d(\fu) + 1$, where $d$ is the grading function on $\Region(P)$ induced by the grading function $d$ on $\Markov$.

We now check that this gluing construction is compatible with that used in the previous section. That is, we verify that $\Diag(P) \subset \Region(P)$, or -- if $\ba = (1,1,1)$ -- that $\Diag(P,v) \subset \Region(P,v)$ for each $v \in \V{P}$. First note that the polygon $P$ is an element of both sets by the definitions of $\Diag(P)$ and $\Region(P)$ respectively. Moreover, assuming that $\fu^i_k \in \Diag(P)$ is an element of $\Region(P)$ for each $i \in \{1,2\}$, the inclusion $\Diag(P) \subset \Region(P)$ (or $\Diag(P,v) \subset \Region(P,v)$) follows from Lemma~\ref{lem:gluing_chambers}. The triangles appearing in this lemma are illustrated in Figure~\ref{fig:gluing_diagrams}.

\begin{lem}
	\label{lem:gluing_chambers}
	Fixing a value of $i \in \{1,2\}$ and $k \in \Z_{\geq 0}$, let $\fu := \fu^i_k(P) \in \Diag(P)$. For each $j \in \{1,2\}$, let $\bar{\fu}^j_1 := \fu^j_1(P_\fu) \in \Diag(P_\fu)$. We have that
	\[
	S_\fu(\bar{\fu}^{3-i}_1) = \fu^i_{k+1}(P).
	\]
\end{lem}
\begin{proof}
	Let $P'$ be the Fano polygon obtained by mutating $P_\fu$ in the edge $E_{3-i}$, while fixing edge $E_i$. $P'$ is related -- by a map $x \mapsto \lambda x + u$, for some $\lambda \in \R_{>0}$ and $u \in N_\R$ -- to both $\bar{\fu}^{3-i}_1$ and $\fu^i_{k+1}(P)$. Moreover the triangles $S_\fu(\bar{\fu}^{3-i}_1)$ and $\fu^i_{k+1}$ both share the same edge with $\fu$ and are hence identical.
\end{proof}

\begin{figure}[ht]
	\includegraphics[scale=0.9]{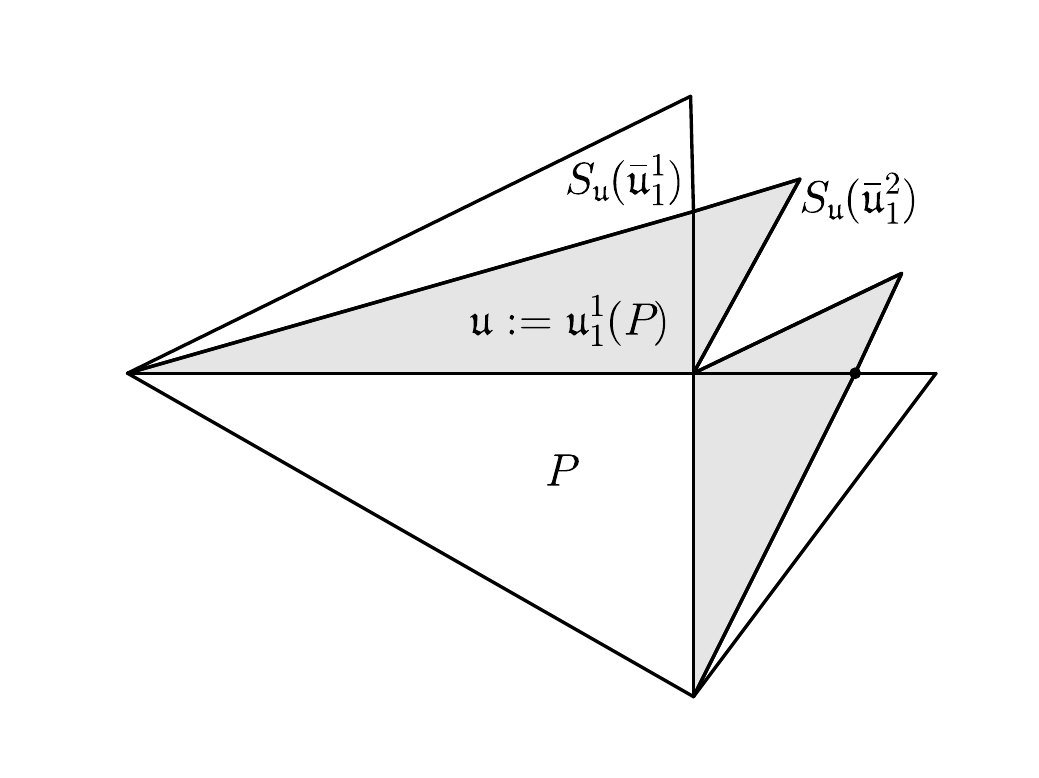}
	\caption{Constructing $\Diag_2(P)$ from $\Diag_1(P)$, and $\Diag_1(P_\fu)$, $\fu \in \Diag_1(P)\setminus \{P\}$.}
	\label{fig:gluing_diagrams}
\end{figure}

We check that $\Region(P)$ defines a polyhedral decomposition of a domain in $\R^2$; that is, that different chambers intersect along faces of each chamber. Given an affine linear map $L \colon \R^2 \to \R^2$ we define $L(\Region(P)) := \{L(\fu) : \fu \in \Region(P)\}$.

We will make use of a version of Lemma~\ref{lem:slope_bound} for $\rho^\star(\Region(P))$, where $\rho$ sends the standard basis in $\Z^2$ to the inner normal vectors of edges $E_1$ and $E_2$ of $P$. Let $\d^i_\infty$ be the rays defined in the previous section, with slopes $m^i_\infty$ for $i \in \{1,2\}$ as shown in Figure~\ref{fig:better_bound}. Fix an edge $E := E_i$ for some $i \in \{1,2\}$ of $P$, and define a sequence $\{\fu_i \in \Region(P) : i \in \Z_{\geq 0}\}$ by setting $\fu_0 = P$, and insisting that $\fu_{k+1}$ is a successor of $\fu_k$ such that the corresponding mutation from $P_{\fu_k}$ fixes the edge $E$. Each $\fu_j$ contains the vertex $v_i$ of $P$, while th remaining vertices of $\fu_i$ converge to a point on $\fr^i$. The triangles $\fu^i_k$ are illustrated in Figure~\ref{fig:better_bound}. We let $r_i$ denote the slope of the ray $\fr^i$.

\begin{figure}
	\includegraphics[scale=1.2]{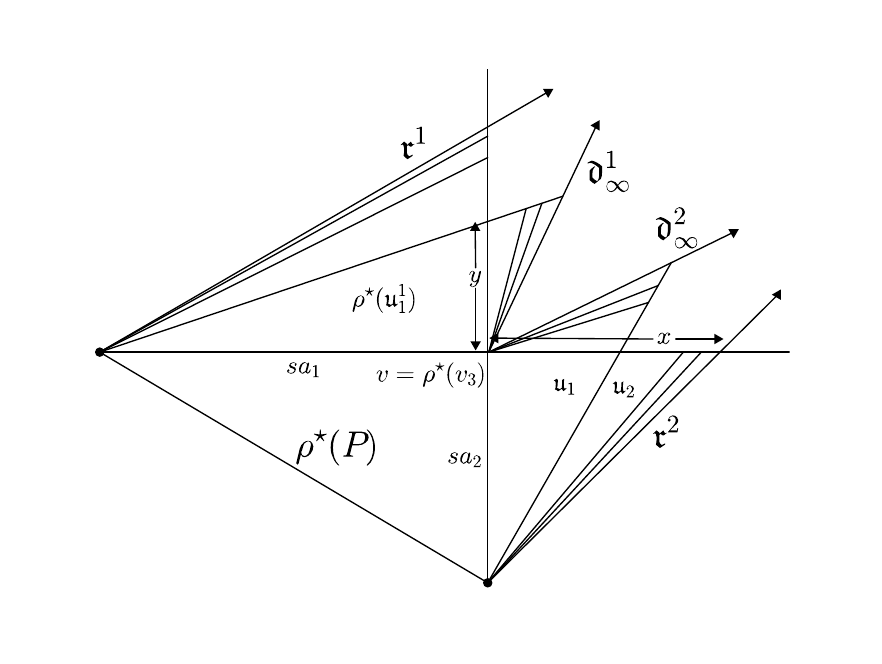}
	\caption{Comparing asymptotes of diagrams}
	\label{fig:better_bound}
\end{figure}

\begin{prop}
	\label{prop:better_bound}
	Given $m^1_\infty$, $m^2_\infty$, $r_1$, and $r_2$ as above, we have that $m^1_\infty > r_i > m^2_\infty$ for $i \in \{1,2\}$.
\end{prop}
\begin{proof}
	Note that -- since $a_1$ and $a_2$ are interchangeable -- we only need to check this inequality for a single value of $i$, in particular we verify this inequality in the case $i=2$. The slope $r_2$ is a limit of slopes of edges of triangles $\fu_i$, indicated in Figure~\ref{fig:better_bound}. We have that $r_2 \leq m_1$ by construction, and $m_1 \leq m^1_\infty$ by Lemma~\ref{lem:slope_bound}. To show $r_2 \geq m^2_\infty$, let $x$ be length of the horizontal segment between $v := \rho^\star(v_3)$ and $\d^2_\infty$ shown in Figure~\ref{fig:better_bound}. Let $y$ be the length of the vertical edge of $\rho^\star(\fu^1_1(P))$. By Lemma~\ref{lem:slope_bound} we have that  $y/sa_1 > m^2_\infty$; hence it is sufficient to show that $r_2 = sa_2/x \geq y/sa_1$. To do this we show that $x \leq sa_1$ and $y < sa_2$.

	Both these inequalities follow from the fact that, given a Markov triple $(a_1,a_2,a_3)$ such that $a_3 \geq a_2$ and $a_3 \geq a_1$, we have $3a_3a_2-a_1 \geq 2a_1$ and $3a_3a_1-a_2 \geq 2a_2$. Considering the diagram $\Diag_1(P)$, and the affine map $A \colon v \mapsto \lambda v + u$ which takes the triangle corresponding to $(a_2,a_3,3a_3a_2-a_1)$ to $\fu^1_1(P)$, $A$ must take the edge with lattice length $(3a_3a_2-a_1)$ to one with lattice length $a_1$; hence, 
	\[
	\lambda = \frac{a_1}{3a_3a_2-a_1} \leq \frac{1}{2},
	\]
	and thus $y = \lambda sa_2 < sa_2$. Repeatedly applying the fact that $\lambda \leq \frac{1}{2}$ for any Fano triangle $P$ associated to a Markov triple, we have that $x \leq \sum_{j \geq 1}\frac{sa_1}{2^j} = sa_1$.	
\end{proof}

Note that, given a lower bound $K$ for the \emph{smallest} value in $\ba$, we can strengthen the bound $\lambda \leq 1/2$. Indeed, 
\[
\lambda = \frac{a_1}{3a_3a_2-a_1} \leq \frac{a_1}{3Ka_1-a_1} = \frac{1}{3K-1}.
\]

\begin{remark}
	\label{rem:small_seg}	
	The fact that $x \leq sa_1$ in the proof of Proposition~\ref{prop:better_bound} implies that the segment obtained by applying ${\rho^\star}^{-1}$ to the horizontal segment between $v$ and $\fr^2$ (extending $\rho^\star(E_1)$) has Euclidean length at most the length of the edge $E_1$.
\end{remark}

Let $R_P$ denote the triangle in $\R^2$ bounded by (initial segments of) the rays $(\rho^\star)^{-1}(\fr^1)$ and $(\rho^\star)^{-1}(\fr^2)$, and the edge $E_3$ of $P$. We also set $R_\fu := S_\fu(R_{P_\fu})$ for any $\fu \in \Region(P)$ or $\Region(P,v)$. If context removes ambiguity we write $R_\fu$ instead of $\rho^\star(R_\fu)$, where $\fu \in \Region(P)$ or $\Region(P,v)$. For example, considering the chamber $\fu = \rho^\star(\fu^1_1)$ indicated in Figure~\ref{fig:better_bound}, the triangular region $R_\fu$ is bounded by edge $\rho^\star(E_1)$ of $\rho^\star(P)$ and the rays $\fr^1$ and $\d^1_\infty$.

\begin{lem}
	\label{lem:containment}
	The region $R_\fu$ is bounded for any $\fu \in \Region(P)$. Moreover if $\fu' \geq \fu$ we have that $R_{\fu'} \subseteq R_\fu$.
\end{lem}
\begin{proof}
	Putting $P$ in standard form, the first part follows from the observation that $\fr^1$ and $\fr^2$ intersect. If the Markov triple $\ba = (a_1,a_2,a_3)$ associated to $P$ is not equal to $(1,1,1)$ this is follows immediately from Proposition~\ref{prop:better_bound} (which shows that $\fr^2$ and $\d^2_\infty$ intersect) after replacing $P$ by the polygon associated to the triple obtained by mutating $\ba = (a_1,a_2,a_3)$ at the maximal value $a_3$. The case $\ba = (1,1,1)$ follows from direct calculation, or, arguing similarly to the proof of Proposition~\ref{prop:better_bound}, $sa_1 = sa_2 = 3$, and $x < 3$. Since in this case $sa_2/x = r_2 = 1/r_1$, we have that $r_2 > 1 > r_1$.
	
	Given a triangle $\fu' \geq \fu$ the fact that $R_{\fu'} \subseteq R_\fu$ follows inductively from Proposition~\ref{prop:better_bound}. Indeed, comparing the regions $R_\fu$ associated to $\fu = \rho^\star(P)$ and $\fu' = \rho^\star(\fu^1_1)$ shown in Figure~\ref{fig:better_bound} we see that, as $m^1_\infty > r_2$, we have that $R_{\fu'} \subseteq R_\fu$.
\end{proof}

\begin{prop}
	\label{prop:no_overlap}
	The set $\Region(P)$ is the set of two dimensional cells of a polyhedral decomposition of a subset of $\R^2$. Moreover for each vertex $v$ of this decomposition, $v$ is either a vertex $v_i$ for $i \in \{1,2\}$ of $P$, or there is a unique $\fu$ such that every triangle in $\Region(P)$ that contains $v$ is an element of $S_\fu(\Diag(P)) := \{S_\fu(\fu') : \fu' \in \Diag(P_\fu)\}$.
\end{prop}
\begin{proof}
	We first show that elements of $\Region(P)$ intersect along faces. By Lemma~\ref{lem:containment} we have that $R_{\fu^i_1(P)} \subseteq R_P$ for each $i \in \{1,2\}$, and hence $\fu^i_1(P) \subset R_P$. Thus, by induction, $\fu \subset R_P$ for all $\fu \in \Region(P)$. Replacing $P$ with $P_\fu$ we have that $\fu' \subset R_\fu$ for any $\fu' \geq \fu$.
	
	Fix elements $\fu$ and $\fu'$ of $\Region(P)$ such that $\fu'$ is a successor of $\fu$ in the partial order on $\Region(P)$. We have that $R_{\fu'} \cap \fu = \fu' \cap \fu$ is a shared edge. Moreover, if $\fu'' \in \Region(P)$ is a successor of $\fu'$ we have that $R_{\fu''} \cap \fu = \fu'' \cap \fu$ is a single shared vertex of $\fu$. If $\bar{\fu} \geq \fu''$, then $\bar{\fu} \cap \fu \subseteq R_{\fu''} \cap \fu$ is either a shared vertex, or empty. In other words, pairs of comparable triangles $\fu,\fu' \in \Region(P)$ intersect along faces.

	Considering the chamber $\fu^i_1(P)$ for each $i \in \{1,2\}$, the intersection $R_{\fu^1_1(P)} \cap R_{\fu^2_1(P)}$ is the vertex $v_3$ of $P$. Suppose now that $\fu \in \Region(P)$ and $\fu' \in \Region(P)$ are incomparable. Replacing $P$ with the polygon $P_{\fu \wedge \fu'}$ we have that $\fu \cap \fu' \subset R_{\fu^1_1(P)} \cap R_{\fu^2_1(P)} \subseteq \{v_3\}$. However, as $\fu$ and $\fu'$ are each contained in $R_{\fu^i_1(P)}$ for distinct values of $i \in \{1,2\}$, $v_3$ must be a vertex of $\fu$ and $\fu'$.
	
	Finally we describe the vertices of triangles appearing in $\Region(P)$. Fix a vertex $v \in \R^2$ of a triangle in $\Region(P)$, and let $k \in \Z_{\geq 0}$ be minimal such that $v \in \fu$ for some $\fu \in A_k \subset \Region(P)$. Assume $v \notin \{v_1,v_2\} \subset \V{P}$, and note that in this case there is a unique $\fu \in A_k$ such that $v$ corresponds to the largest value of the Markov triple associated to $P_\fu$. We now classify the chambers $\fu'$ in $\Region(P)$ which contain $v$. We consider a number of cases: First, if $\fu' < \fu$, $v \notin \fu'$ by hypothesis, and the fact that triangles in $\Region(P)$ intersect in faces. If $\fu' \in S_\fu(\Diag(P_\fu))$, then $v \in \fu'$ by the definition of $\Diag(P_\fu)$, see Figure~\ref{fig:adding_rays}. If $\fu' > \fu$ and $\fu' \notin  S_\fu(\Diag(P_\fu))$, then $v \notin \fu'$. Indeed, replace $P$ with $P_\fu$ and observe that any $\fu_1 \in \Diag(P)$ has a unique successor $\fu_2 \notin  \Diag(P)$. Note that $v \notin R_{\fu_2}$, and hence $v$ is not contained in any triangle $\fu_3 \geq \fu_2$. Finally, if $\fu$ and $\fu'$ are incomparable, $v \notin \fu'$ as $v \in \Int(R_\fu)$, but $\fu' \cap \fu \subset \partial R_\fu$.
\end{proof}

For each $\fu \in \Region(P)$ (or $\Region(P,v)$), let $v_\fu$ be the unique vertex of $\fu$ not contained in any $\fu' < \fu$. We define the \emph{support} of $\Region(P)$ to be the set 
\[
\{x \in \R^2 : \text{$x \in \fu$ for some $\fu \in \Region(P)$}\}.
\]
\begin{lem}
	\label{lem:edges_in_region}
	The union of edges of triangles in $\Region(P) \setminus \{P\}$ is the intersection of a collection of rays in $\R^2$ with the support of $\Region(P)$.
\end{lem}
\begin{proof}
	By induction we see that every edge $E$ of a triangle $\fu \in \Region(P)\setminus \{P\}$ is either an edge of a triangle $\fu' < \fu$ or contains $v_\fu$. The edges of $\fu$ incident to $v_\fu$ are extended by edges of the two triangles $S_\fu(\fu^i_1(P_\fu))$ for $i \in \{1,2\}$. If $E$ is contained in $\fu'$ it contains $v_{\fu'}$ (otherwise $E$ is an edge of three regions: $\fu$,  $\fu'$, and some $\fu'' < \fu'$), and this edge is extended by the triangle $S_{\fu'}(\fu^i_1(P_{\fu'}))$ not equal to $\fu$.
\end{proof}

\section{Background On the Gross--Siebert Algorithm}
\label{sec:GS_background}

In this section we briefly recall the main definitions, results, and notation used in the Gross--Siebert algorithm; we refer to~\cite{GS1} for full details. The Gross--Siebert algorithm takes as input a `discrete' part and a `continuous' part. The discrete part consists of a tuple $(B,\sP,s,\varphi)$, where:
\begin{enumerate}
	\item $B$ is an integral affine manifold with singularities,
	\item $\sP$ is a polyhedral decomposition of $B$,
	\item $s$ is open gluing data, defined in~\cite{GS1} and,
	\item $\varphi$ is a polarisation, a multi-valued piecewise linear function on $B$.
\end{enumerate}

\begin{definition} \label{def:affine_manifold}
	An integral affine manifold with singularities is a topological manifold $B$, an open dense submanifold $B_0 \subset B$ such that $\Delta := B \setminus B_0$ has codimension at least two, and an atlas on $B_0$ such that all transition functions lie in the group of $\Z$-affine functions $\Z^2 \rtimes \GL(2,\Z)$.
\end{definition}

Observe that in $B_0$ linear objects (lines, polyhedra) are well defined, and the notion of polyhedral decomposition is well-defined. The continuous part consists of a collection of \emph{slab functions} $f_\tau$. Given a codimension one cell $\tau$ of $\sP$; a slab function is a section of a line bundle on the toric variety defined by the normal fan of $\tau$. This line bundle is determined by the monodromy around the singular locus in the affine structure on $B$, we refer to \cite{GS1} for further details.

The Gross--Siebert algorithm, in the form described in \cite[Chapter~$6$]{TropGeom}, iteratively constructs a collection of rays $\sS$, called a \emph{structure}, together with a notion of \emph{order} for each ray. The set of rays with order at most $k$ form finite sets $\sS[k]$. For each $k \in \Z_{\geq 0}$ one defines the set $\Chambers(\sS,k)$ to be maximal cells of an auxiliary polyhedral decomposition of $\breve{B}$. The conditions this decomposition needs to satisfy are given in \cite[Definition~$6.24$]{TropGeom}; roughly, rays in $\sS[k]$ are unions of edges of the decomposition. Starting from an initial structure, the Gross--Siebert algorithm constructs a \emph{compatible} structure, which can be used to build a toric degeneration. Indeed, given a compatible structure Gross--Siebert define a functor $F_k$ from a category -- called $\underline{\text{Glue}}(\sS,k)$ -- defined from $\sP_k$ to the category of rings. Objects of $\underline{\text{Glue}}(\sS,k)$ are certain triples $(\omega,\tau,\fu)$ where $\omega$ and $\tau$ are strata of $\sP$ and $\fu \in \Chambers(\sS,k)$. The ring $F_k(\omega,\tau,\fu)$ is denoted $R^k_{\omega,\tau,\fu}$, and we refer to see \cite{GS1} or \cite[p.$277$]{TropGeom} for a complete definition. In general, the ring $R^k_{\omega,\tau,\fu}$ is a localisation of a quotient of the polynomial ring $\field[P_{\varphi,\omega}]$. The semigroup $P_{\varphi,\omega}$ is contained in an extension of the lattice $\Lambda$ of integral tangent vectors at a point in $\omega$ by $\Z$. We refer to elements $m$ of $P_{\varphi,\omega}$ as \emph{exponents}, and let $\bar{m}$ denote their projection to $\Lambda$. It is proved in~\cite{GS1,TropGeom} that, taking the inverse limit over this system of rings, a compatible structure defines a flat formal deformation of the reducible union of toric varieties whose moment polytopes are the maximal cells of $\sP$.

The main tool used in~\cite{GS1} to construct the walls of $\sS[k]$ is the notion of \emph{scattering diagram}. The only examples of scattering diagrams we shall use are the most basic studied and are studied in detail in~\cite{GPS,GHKK}, to which we refer for details.

\begin{figure}
	\includegraphics[scale=1.3]{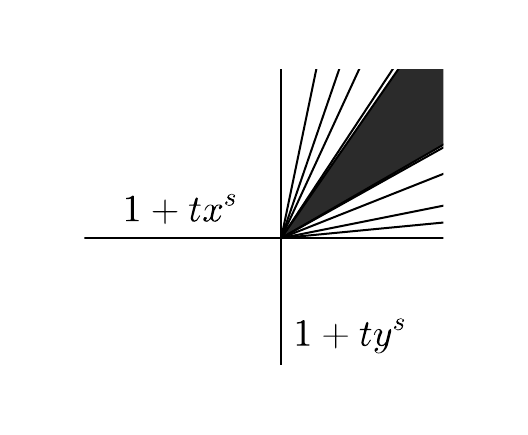}
	\caption{Rays of a scattering diagram $\fD_k(s)$.}
	\label{fig:scattering_1}
\end{figure}

Given $s \in \Z_{>0}$ we define the (initial) scattering diagram,
	\[
	\fD_0(s)  = \{ (\R(1,0), 1+tx^s) , (\R(0,1), 1+ty^s) \}.
	\]
Scattering diagrams `localise' the problem of finding a compatible structure, replacing it with the problem of achieving a certain consistency condition near each scattering diagram by adding new rays for each $k \in \Z_{\geq 0}$. The constructive proof that any scattering diagram can be made consistent by adding rays is due to Kontsevich--Soibelman~\cite{KS06}. The rays added to $\fD_0(s)$ by the scattering process admit a periodicity, resulting in a recursive formula identical to that found in Lemma~\ref{lem:recursive_rays}. Let $\fD_k(s)$ denote the scattering diagram after adding rays to order $k$ to $\fD_0(s)$, and let $\fD(s)$ denote the limiting scattering diagram as $k \to \infty$. Figure~\ref{fig:scattering_1} illustrates the supports of the rays of $\fD(s)$.

Fix a Fano triangle $P$ associated to a Markov triple $\ba$, and let $s$ denote the wedge product of the inner normal vectors to edges $E_1$ and $E_2$ (using the notation for edges of $P$ introduced in \S\ref{sec:fano_polygons}). Put $P$ in standard form with respect to $v := \rho^\star(v_3)$.

\begin{prop}
	\label{prop:rays_and_triangles}
	The set of rays $\left\{\d^i_j : i \in \{1,2\}, j \in \Z_{>0} \right\}$ given in Definition~\ref{dfn:rays_from_edges} is equal to:
	\[
	\left\{\d : (\d,f_\d) \in \fD(s), m_\d > \frac{s+\sqrt{s^2-4}}{2} \textrm{ or } m_\d < \frac{s-\sqrt{s^2-4}}{2} \right\},
	\]
	where $s$ is the determinant of the map $\rho$ and $m_\d$ is the slope of the ray $\d$.
\end{prop}
\begin{proof}
	Each of the rays $\d^i_j$ is generated by a vector $v^i_j$ which are defined via the recursive formula
	\[
	v^i_{j-1} + v^i_{j+1} = sv^i_j.
	\]
	We show that the rays -- with slopes satisfying the given bounds -- in the support of $\fD(s)$ satisfy the same recursion relation. This follows from the invariance of the scattering diagram under cluster mutations proven in \cite{GHKK}. In particular in \cite[Example~$1.15$]{GHKK} it is observed that every ray -- between the specified bounds -- is generated by alternately applying the linear transformations with matrices 
	\begin{align*}
	 S_1 := \begin{pmatrix} -1 & s \\ 0 & 1 \end{pmatrix} && S_2 := \begin{pmatrix}  1 & 0 \\ s & -1  \end{pmatrix}
	\end{align*}
	to the vectors	$(1,0)$ and $(0,1)$. Note that there is a sign difference between our transformations and those considered in \cite{GHKK}, as our rays lie in the first quadrant of $\R^2$ rather than the fourth quadrant. Rays in the support of $\fD(s)$ are generated by vectors $\tilde{v}^i_j$ for $i \in \{1,2\}$, $j \in \Z_{>0}$, where:
	\begin{itemize}
		\item  $\tilde{v}^1_1 := (0,1)^T$, $\tilde{v}^2_1 := (1,0)^T$,
		\item  $\tilde{v}^2_{j+1} := S_1\tilde{v}^1_k$, and,
		\item  $\tilde{v}^1_{j+1} := S_2\tilde{v}^2_k$.
	\end{itemize}
	It is easily verified that $\tilde{v}^i_j = \begin{pmatrix}0&1 \\1&0\end{pmatrix}\tilde{v}^{3-i}_j$, and hence, writing $\tilde{v}^1_j = \begin{pmatrix}y_{j-1}\\ y_j\end{pmatrix}$, 
	\begin{align*}
	\tilde{v}^1_{j+1} &= \begin{pmatrix}0&1 \\1&0\end{pmatrix} \begin{pmatrix}  -1 & s \\ 0 & 1  \end{pmatrix} \begin{pmatrix}
	y_{j-1} \\ y_j
	\end{pmatrix} \\ &= \begin{pmatrix}
 	  y_j \\ sy_j-y_{j-1}
 	 \end{pmatrix}.
	\end{align*}
	Thus $y_{j+1} + y_{j-1} = sy_j$, and hence $y_j = x_j$, and $\tilde{v}^i_j = v^i_j$ for all $i$ and $j$, as required.
\end{proof}

\begin{remark}
	While the association of this set of rays with the triangles of the previous section seems somewhat mysterious, it is in fact tautological given the connection that both concepts have with cluster algebras. The relationship between combinatorial mutation and cluster algebras is explored in~\cite{KNP15}, and the deep connections between scattering diagrams and cluster algebras are explored in~\cite{GHKK}.
\end{remark}

The following is a well-known expectation in the theory of scattering diagrams, although we do not know a reference for a proof. This conjecture is referred to as an expectation in \cite{GPS,GHKK}, and the statement is assumed in \cite[p.$5$]{B17}.

\begin{conj}
	\label{conj:dense_with_rays}
	Every ray $\d$ in the positive quadrant with rational slope between $\frac{s - \sqrt{s^2-4}}{2}$ and $\frac{s + \sqrt{s^2-4}}{2}$ appears in the support of the scattering diagram $\fD_k(s)$ -- that is, appears with non-trivial function $f_\d$ -- for some $k$.
\end{conj}

As well as the notion of a scattering diagram we will utilize the notion of a \emph{broken line} from \cite{G09,CPS11}. These will provide an enumerative interpretation of the Laurent polynomials mirror to $\P^2$ as described in Theorem~\ref{thm:P2_classification}. The notion of broken line is very close to that of a \emph{tropical disc}: broken lines can \emph{bend} on the walls of a scattering diagram and one can canonically complete these bends so that the resulting object is a tropical curve \emph{with stops} (following the terminology of \cite{Nishinou:discs}). For more details see \cite[Lemma~$5.4$]{CPS11}.

The idea of calculating a superpotential tropically, utilising broken lines in the affine manifold, was first explored in \cite{CPS11}. In \S\ref{sec:proof} we show that there is a domain $U$ in the dual intersection complex $\breve{B}_{\P^2}$ of a toric degeneration of $(\P^2,E)$ such that the tropically defined superpotential is equal to the family of Laurent polynomials described in \cite{Triangles}.

In an ideal setting tropical curves should be the `spines' of images of holomorphic curves under a special Lagrangian torus fibration. Tropical discs are similar, but now the curve has boundary and a `stop' is introduced where the tropical disc terminates. For a more detailed discussion of this point see \cite{Nishinou:discs,CPS11,G09}. 

\begin{definition}
	\label{dfn:broken_line}
	Fixing a value of $k \in \Z_{\geq 0}$, a \emph{broken line} is a proper continuous map
	\[
	\beta \colon (-\infty,0] \rightarrow B
	\]
	with `bends' at a sequence of points $-\infty = t_0 < t_1 < \cdots < t_r = 0$ such that $\beta|_{(t_j,t_{j+1})}$ is an affine map with image disjoint from the rays of $\sS[k]$.
	
	The broken line $\beta$ carries a sequence of monomials $a_jz^{m_j}$ such that $\beta'(t) = \bar{m}_j$. Fixing a value of $i \in \{0,\ldots,r\}$, let $\fu$ and $\fu'$ denote the distinct chambers of $\sS[k]$ containing $\beta(t_i-\epsilon)$ and $\beta(t_i+\epsilon)$ for sufficiently small $\epsilon \in \R_{>0}$ respectively. The monomial $a_jz^{m_j}$ defines a unique element in the ring $R^k_{\tau,\tau,\fu}$ and the wall-crossing formula $\theta_{\fu, \fu'}$ defines a collection of monomials with order $\leq k$; these are the \emph{results of transport} of $a_jz^{m_j}$.
	
	We also insist that $a_1 = 1$, and that there is an unbounded 1-cell of $\sP$ parallel to $\bar{m}_1$ for which $m_1$ has order zero.
\end{definition}



Given a general\footnote{This is a generic condition, see \cite[Proposition~$4.4$]{CPS11} and \cite[Definition~$4.6$]{CPS11}.} point $p \in B$, denote the set of broken lines $\beta$ with $\beta(0) = p$ by $\mathfrak{B}(p)$. For a given structure $\sS[k]$ on $B$, and a chamber $\fu$ such that $p \in \fu$, we can produce the \emph{superpotential at order} $k$ as an element of $R^k_{\omega,\tau,\fu}$, taking
\[
W^k_{\omega,\tau,\fu}(p) = \sum_{\beta \in \mathfrak{B}(p)}a_\beta z^{m_\beta}
\]

In \cite{CPS11} the authors obtain various results for $W^k_{\omega,\tau,\fu}(p)$, two of which we shall utilize in \S\ref{sec:proof}.
\begin{enumerate}
	\item The superpotential $W^k_{\omega,\tau,\mathfrak{u}}(p)$ is independent of the choice of $p \in \mathfrak{u}$~\cite[Lemma 4.7]{CPS11}.
	\item The superpotentials are compatible with changing strata and chambers~\cite[Lemma 4.9]{CPS11}.
\end{enumerate}

The content of the second point here is that, applying a change of chamber map to the superpotential, one obtains
\[
\theta_{\fu,\fu'}(W^k_{\omega,\tau,\fu}) = W^k_{\omega,\tau,\fu'}
\]
where we have suppressed the dependence of  $W^k_{\omega,\tau,\fu}(p)$ on $p$ using the first point. This formula implies that the superpotential changes by algebraic mutation discussed in Remark~\ref{rem:alg_mutation}. To see this, we need to compare the rings $R^k_{\omega,\tau,\fu}$ and $\field[M]$ of which the respective superpotentials are elements. In \S\ref{sec:proof} we shall find that the superpotential $W^k_{\omega,\tau,\fu}$ is, in the terminology of \cite{CPS11}, \emph{manifestly algebraic} in a subset $V \subset \breve{B}$. The main consequence of this definition is that the limit $W$ of polynomials $W^k$ is also polynomial. Following \cite[p.$277$]{TropGeom}, $R^k_{\omega,\tau, \fu}$ is a localisation of the ring
\[
\field[P_{\varphi,\omega}]/I^k_{\omega,\tau,\sigma_\fu},
\]
where $\sigma_\fu$ is defined in \cite[p.$276$]{TropGeom} and $I^k_{\omega,\tau, \sigma_\fu}$ is defined in \cite[p.$265$]{TropGeom}. Thus, for sufficiently large values of $k$, the lift of $W^k_{\omega,\tau,\fu}$ to $\field[P_{\varphi,\omega}]$ is independent of $k$. In fact, by Theorem~\ref{thm:superpotential}, the decompositions $\sP_k$ can be chosen so that $\fu$ does not undergo subdivision for large values of $k$. Taking the projection $\field[P_{\varphi,\omega}] \rightarrow \field[M]$ induced by setting $t=1$, we can represent $W$ as a single Laurent polynomial. We summarise the definition of the wall crossing formula appearing in \cite{CPS11} and \cite{G09} in the following Lemma.

\begin{lem}
	The wall crossing formula
	\[
	\theta^k_{\fu,\fu'}(z^m) = z^m \prod{f^{\langle n,\bar{m} \rangle}_\d}
	\]
	defines a birational map $\theta^k_{\fu,\fu'} \colon \field(M) \rightarrow \field(M)$. If there is only a single ray supported on $\d$ and $f_\d = 1+c_mz^m$ for some exponent $m$ then the birational map $\theta^k_{\fu,\fu'}$ is an \emph{algebraic mutation} with factor polynomial $(1+c_mz^{\bar{m}})$.
\end{lem}

Thus the result of \emph{crossing a wall} is that the function recorded at the base point, viewed simply as a Laurent polynomial, undergoes a birational change of variables which is precisely the mutation with factor given by the line segment in the direction of the wall. This is an essential ingredient in the proof of Theorem~\ref{thm:superpotential_2} since it will allow us to compute the superpotential in every chamber from a calculation of broken lines in a single chamber.
\section{The Affine Manifold $\breve{B}_{\PP^2}$}
\label{sec:aff_mfld}

We now consider the affine structure on the dual intersection complex $\breve{B}_{\P^2}$ for a toric degeneration of $\P^2$. This is described in \cite[Example~$2.4$]{CPS11}. The authors of \cite{CPS11} consider the affine structure on the intersection complex and dual intersection complex of a so-called \emph{distinguished toric degeneration}. Given the pair $(\mathbb{P}^2,E)$ for a smooth genus one curve $E$, a distinguished toric degeneration will give an intersection complex as shown in Figure~\ref{fig:intersection_complex}, as shown in the proof of Theorem~$6.4$ in \cite{CPS11}.

For a precise definition of the discrete Legendre duality between $B_{\P^2}$ and $\breve{B}_{\P^2}$ see~\cite{GS1}. Rather than provide this definition here we will describe $\breve{B}_{\P^2}$ as an affine manifold. The affine manifold $B_{\P^2}$ associated to the \emph{intersection} complex is shown in Figure~\ref{fig:intersection_complex}. The affine structure on $\breve{B}_{\P^2}$ is such that the three `outgoing' unbounded $1$-cells of $\sP$ are parallel to each other, the dual condition to the requirement that $B_{\P^2}$ have smooth (flat) boundary. In particular, as a topological manifold, $\breve{B}_{\P^2}$ is isomorphic to $\R^2$ and its affine structure contains three focus-focus singularities. This affine structure is described in \cite{CPS11}, and is illustrated in Figure~\ref{fig:initial_lines}.

Charts are described by cutting along the invariant lines of each focus-focus singularity. More precisely, letting $P$ denote the central triangular region shown in Figure~\ref{fig:initial_lines}, we fix six charts
\[
\Phi_{v_1,v_2}\colon U_{v_1,v_2} \to \R^2
\]
on $\breve{B}_{\P^2}$, where $v_1, v_2 \in \V{P}$ and $v_1 \neq v_2$. The domain $U_{v_1,v_2} \subset \breve{B}_{\P^2}$ is formed by removing a ray emanating from each of the three focus-focus points such that none of these three rays contains $v_1$ and exactly one contains $v_2$, and taking the connected component of the resulting space which contains $P$. An example of such a chart is shown in Figure~\ref{fig:gluing_things}(a). The map $\Phi_{v_1,v_2}$ identifies $U_{v_1,v_2}$ with the domain in $\R^2$ formed by removing rays from each singular point from $P$ -- regarded as a subset of $\R^2$ -- and restricting to the connected component containing $P$. This identification is made such that the transition function $\Phi_{v_2,v_1}\circ\Phi^{-1}_{v_1,v_2}$ is the identity on $P$, and conjugate to $
\begin{pmatrix}
1 & 1 \\ 0 & 1
\end{pmatrix}$
on the unique half space $H  \subset \R^2$ such that $v_1,v_2 \in \partial H$.


\begin{remark}
	Following the work of Gross--Hacking--Keel for log Calabi--Yau manifolds \cite{GHK1,GHK2} one might attempt to consider the affine manifold obtained by regarding all the singularities of $\breve{B}_{\P^2}$ as lying at the origin, which would -- in the setting described in in \cite{GHK1,GHK2} -- play the role of $U^{trop}$, for a log Calabi--Yau $U$. However in this case $U$ does not have \emph{maximal boundary}: the resulting affine manifold is a single ray and does not fit easily into this framework.
\end{remark}

Following the Gross--Siebert program~\cite{GS1}, we endow the 1-cells $\tau$ of $B_{\P^2}$ supporting $\Delta$ with slab functions $f_\tau$ defining a log structure on a reducible union of toric varieties. We shall make the standard choices of normalisation so that $f_\tau$ is $(1+z^{m})$ where $m$ is an exponent such that $\bar{m}$, a integral lattice tangent vector to $\breve{B}_{\P^2}$ at a point on $\tau$, is primitive, lies parallel to $\tau$ and toward the focus-focus singularity. We use the following construction to describe the set of (supports of) rays which appear in $\sS[k]$ for some $k \in \Z_{\geq 0}$.

\begin{figure}
	\includegraphics[scale = 0.8]{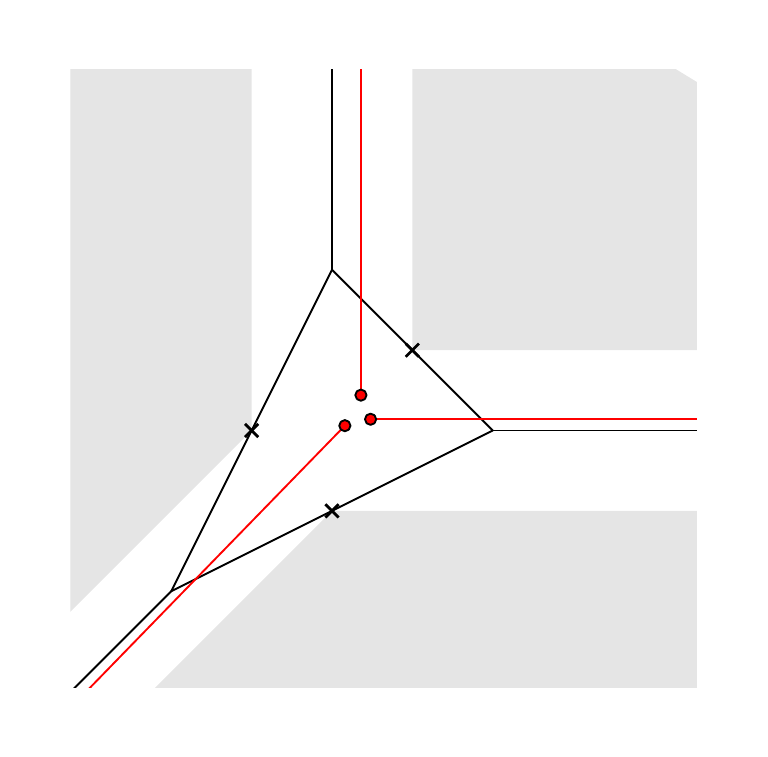}
	\caption{Broken lines in the central region}
	\label{fig:initial_lines}
\end{figure}

\begin{construction}
	\label{cons:support}
	Let $\Rays_1$ be the set of six rays emanating from $\Delta$ parallel to an edge of $P$; these rays intersect at the vertices of $P$, and assume that we have constructed sets $\Rays_k$. For each intersection point $v$ of rays $\d_1$ and $\d_2$ in $\Rays_k$, let $A_v \colon \R^2_{\geq 0} \to \breve{B}_{\P^2}$ be the map taking $(0,0) \mapsto v$ and the vectors $(1,0)$, $(0,1)$ to the (integral) tangent directions of the rays $\d_1$ and $\d_2$. Set
	\[
	\Rays_{k+1} := \Rays_k {\textstyle\coprod} \coprod_{\substack{v \in \Ima\d_1 \cap \Ima\d_2 \\ \d_1,\d_2 \in \Rays_k}} \left\{A_v(\d) : (\d,f_\d) \in \fD(s)\right\},
	\]
	where $s$ is the (absolute value of the) wedge product of the direction vectors of $\d_1$ and $\d_2$, and we assume that $\d_1 \neq \d_2$. It follows from \cite[Theorem~$6.49$]{TropGeom} that the set $\bigcup_{i \geq 0}{\Rays_i}$ is the set of supports of rays in the compatible structure $\sS$.
\end{construction}

We use \cite[Corollary~$6.8$]{CPS11} to compute the tropical superpotential in $\breve{B}_{\P^2}$. This states that, for a base point in the interior of the bounded cell of $\sP$, the superpotential for this structure is given by the usual Givental/Hori--Vafa superpotential:
\[
W = x+y+\frac{1}{xy},
\]
see Figure~\ref{fig:initial_lines}. By \cite[Lemma~$4.9$]{CPS11}, this calculation determines the superpotential in every other chamber, using the wall-crossing formula $\theta_{\fu,\fu'}$ to pass between chambers.


\section{Proof of Theorems $1.1$ and $1.2$}
\label{sec:proof}

Throughout this section we fix the Fano triangle $P := \conv{(1,0),(0,1),(-1,-1)} \subset N_\R$. We also use $P$ to denote the central region in the affine manifold $\breve{B} := \breve{B}_{\P^2}$, see Figure~\ref{fig:initial_lines}. Note that these polygons are identified by the chart $\Phi_{v,v'}$ for any distinct pair of vertices $v,v' \in \V{P}$. We first construct the subset $V \subset \breve{B}$ which appears in the statement of Theorem~\ref{thm:superpotential}. We then identify the rays of the structure $\sS$ which intersect $V$ in a line segment with edges of triangles in $\Region(P,v)$ for some $v \in \V{P}$. We use this to show that we can choose polyhedral decompositions $\sP_k$ such that for any $v \in \V{P}$ and $\fu \in \Region(P,v)$, $\fu \in \sP_k$ for all sufficiently large $k$.

Fixing a vertex $v$ of the polygon $P$, set
\[
\cT_v := \left\{\cl\left(\Phi_{v,v'}^{-1}(\fu^\circ)\right) : \fu \in \Region(P,v)\right\}.
\]
where $v' \in \V{P}$ is different from $v$. We also define $V_v := \bigcup_{\fu \in \cT_v}\fu \subset \breve{B}$; the support of $\cT_v$. Note that, since the transition function between the two possible charts $\Phi_{v,v'}$ acts as the identity on $V_v$, the choice of vertex $v'$ does not affect the subset $V_v$ or any triangle $\fu \in \cT$.

\begin{remark}
	Taking the interior $\fu^\circ$ of $\fu \in \Region(P,v)$ is necessary, since $\Region(P,v)$ is not contained in the image of $U_{v,v'}$. We note however that the complement of $\Ima \Phi_{v',v}$ is contained in a pair of edges of $P$, and does not intersect the interior of any $\fu \in \Region(P,v)$.
\end{remark}

We define $\cT := \bigcup_{v \in \V{P}}{\cT_v}$, and set
\[
V := \bigcup_{v \in \V{P}}{V_v} \subset \breve{B},
\]

This definition of $\cT$ identifies chambers in different set $\cT_v$. In particular the canonical map $I \colon \coprod_{v \in \V{P}}{\cT_v} \to \cT$, is a two-to-one function onto all elements in $\cT\setminus \{P\}$. This follows from the fact -- see Lemma~\ref{lem:V_1} -- that $I$ identifies three pairs of the six triangles $\fu^i_1 \in \Diag_1(P,v)$, obtained by varying $i \in \{1,2\}$ and $v \in \V{P}$. Given a vertex $v$ of $P$, consider the diagram $\Diag_1(P,v)$, and define the triangle
\[
T_{i,v} := \cl(\Phi^{-1}_{v,v'}({\fu^i_1}^\circ)),
\]
where $v' \neq v$ is a vertex of $P$. This construction produces six triangles $T_{i,v} \in \cT$. We now show that three pairs of these triangles are identical.

\begin{figure}[H]
	\centering
	\subfigure[Gluing a diagram in $U_{v,v'}$]{%
		\includegraphics[scale=0.85]{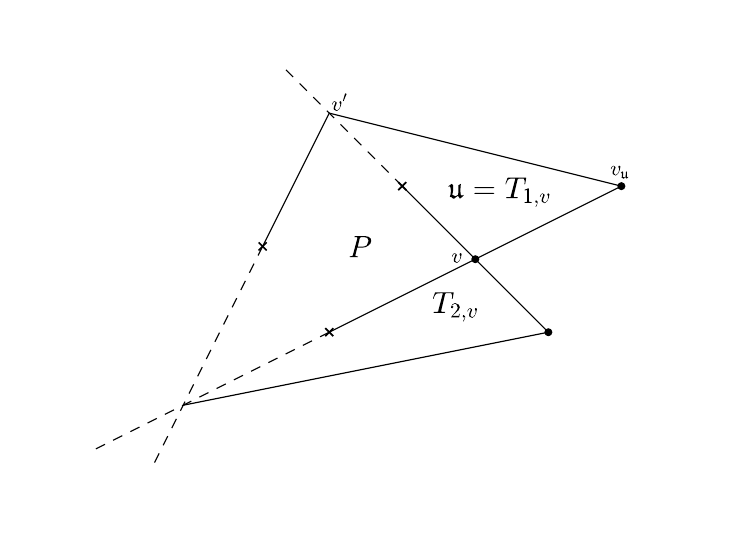}
		\label{fig:chart_1}
	}
	\quad
	\subfigure[Gluing in $\Diag_1(P_\fu)$]{%
		\includegraphics[scale=0.85]{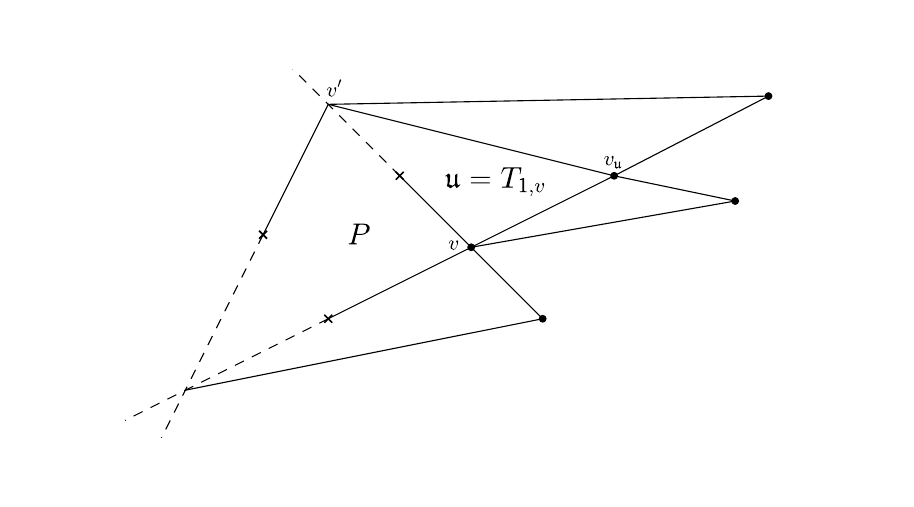}
		\label{fig:chart_2}}
	\subfigure[Gluing in $\Diag_1(P_{\fu'})$]{%
		\includegraphics[scale = 0.9]{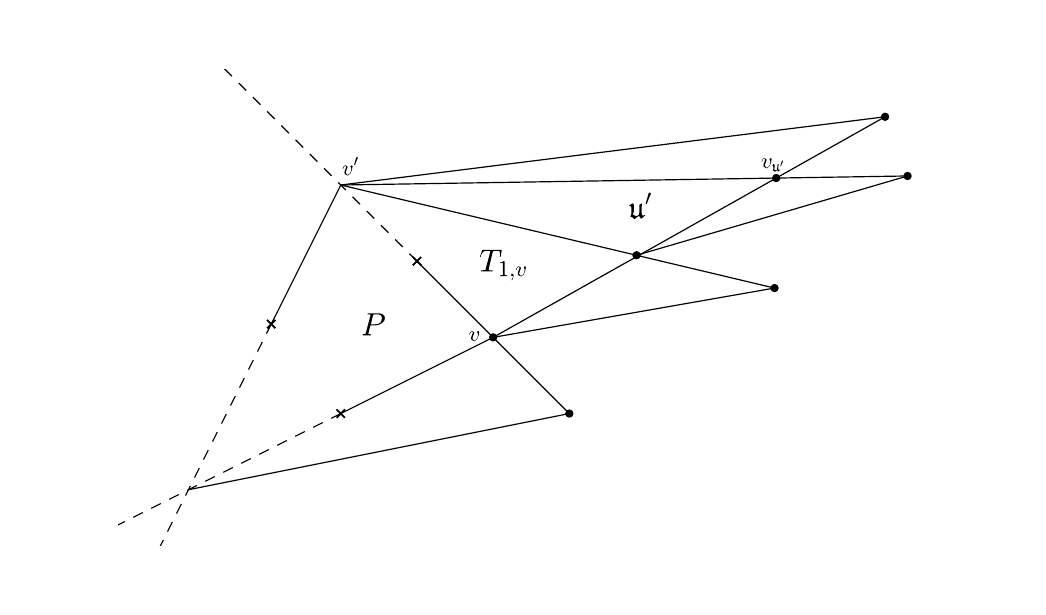}
		\label{fig:chart_3}}
	\caption{Building up $\Chambers(\sS,k)$ using polygon mutation \label{fig:gluing_things}}
	
\end{figure}

\begin{lem}
	\label{lem:V_1}
	Let $v$ and $v'$ be vertices of $P$, and fix $i \in \{1,2\}$ such that $v' \in T_{i,v}$; then $T_{v,i} = T_{v',3-i}$ as subsets of $\breve{B}$.
\end{lem}
\begin{proof}
	Consider the chart $U_{v,v'}$ of $\breve{B}$. The triangle $T_{i,v}$ is formed by an edge $E_1$ of $P$, the continuation $E_2$ of the other edge of $P$ which meets $v$, and a third edge $E_3$. This is illustrated in Figure~\ref{fig:region_2}. Moreover let $E'_3$ denote the edge of $P$ disjoint from $v$. Without loss of generality we may assume that -- in a chart on $\breve{B}$ -- $v = (1,0)$, and $v' = (0,1)$. Applying the formula given in \eqref{eq:PL_map}, the direction of the edge $E_3$ is obtained from the direction of $E'_3$ by applying the linear transformation with matrix
	$\begin{pmatrix}
	0 & 1 \\
	-1 & 0
	\end{pmatrix}$. That is, by applying the linear transformation which fixes the linear subspace parallel to $E_1$ and sends $(0,1)^T \mapsto (1,0)^T$. The same linear map appears as the restriction of the transition function $\Phi_{v',v}\circ\Phi^{-1}_{v,v'}$ to the connected component of  $U_{v,v'}\cap U_{v',v}$ which contains the interior of $T_{v,i}$. Thus the direction of $E_3$ in the chart $U_{v',v}$ is same as the direction of the edge $E'_3$. Thus $E_3$ and $E_1$ are both edges of the triangle $T_{v',j}$ for some $j \in \{1,2\}$. The index $j$ depends on the binary choice of edges $E_1$ and $E_2$ in the construction of each $\Diag_1(P,v)$, but these choices can be made so that $j = 3-i$ for every $v$ and $i$.
\end{proof}

\begin{remark}
	Note that the set $\cT$ is partially ordered, and in fact there is an order preserving bijection between $\cT$ and the trivalent graph $\Markov$; hence $\cT$ is a graded meet-semilattice. We let $\fu \wedge \fu'$ denote the infimum of elements $\fu$ and $\fu'$ of $\cT$.
\end{remark}

Given an edge $E = \conv{v,v'}$ of $P$, let $\fu$ be the unique successor to $P$ in $\cT$ which contains $E$. Let $R_E$ be the image of $R_\fu$ -- defined in \S\ref{sec:gluing_diagrams} -- under $\Phi^{-1}_{v,v'}$. Moreover the subset $U_{v,v'}$ contains all of $R_E$, except for part of the edge $E$ of $P$, see Figure~\ref{fig:gluing_things}(a). Note that any point in $V$ is contained in $R_E$ for some edge $E$ of $P$. Moreover, given a pair of edges $E,E'$ of $P$, $R_E\cap R_{E'} \subset \V{P}$.

\begin{lem}
	\label{lem:edges_form_rays}
	The union of all edges of triangles in $\cT$ is a set of rays in $V$.
\end{lem}
\begin{proof}
	By Lemma~\ref{lem:edges_in_region} we have that triangles $\fu \in \Region(P,v)\setminus\{P\}$ form a set of rays in the support of $\Region(P,v)$ in $\R^2$. Given such a ray $\d$, we have that $\Phi_{v,v'}(\Ima\d\setminus P) \subset R_E$ for some $E = \conv{v,v'}$ of $P$. Thus, possibly excluding a segment of an edge of $P$, $\Phi^{-1}_{v,v'}$ identifies $\d$ with a ray in $\breve{B}$. However, each edge of $P$ is also a line segment in $\breve{B}$ (meeting $\Delta \subset \breve{B}$ in a monodromy invariant direction). The process of extending rays by gluing triangles is illustrated in Figure~\ref{fig:gluing_things}.
\end{proof}

\begin{figure}[H]
	\makebox[\textwidth][c]{\includegraphics[scale=0.7]{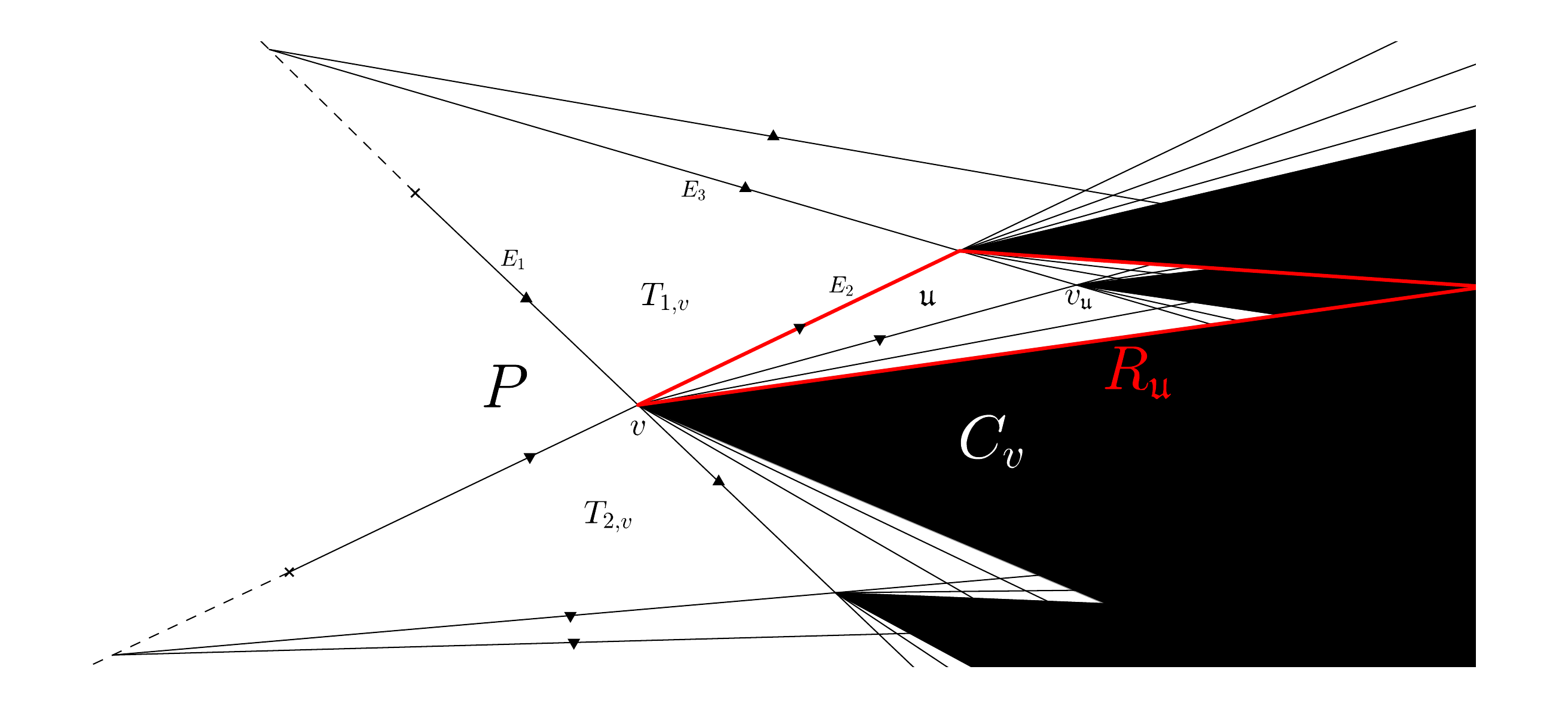}}%
	\caption{Triangles formed by the structure $\sS$.}
	\label{fig:region_2}
\end{figure}

\begin{remark}
	\label{rem:outgoing}
	For any $\fu \in \cT \setminus \{P\}$, let $v_\fu$ denote the vertex of $\fu \in \cT$ which is not contained in $\fu'$ for all $\fu' < \fu$, see Figure~\ref{fig:region_2}. Orienting the edges of $\Diag_1(P_\fu)$ for some $\fu \in \cT$ -- as shown in Figure~\ref{fig:region_2} -- we see that the vertex $v_\fu$, corresponding to the maximal value in the Markov triple associated to $P_\fu$, meets two \emph{incoming} rays (edges of $\fu$), and every other ray incident to $v$ is \emph{outgoing}.
\end{remark}

As well as the subspace $V \subset \breve{B}$, we define a subspace $W \subset \breve{B}$ which will correspond to the subset of $\breve{B}$ dense with rays. Let $J$ be the set of \emph{joints} in $V$: the set of points $p \in \breve{B}$ such that $p$ is a vertex of a triangle $\fu \in \cT$. For each $p \in J$ we have that either $p = v_\fu$ for some $\fu \in \cT \setminus\{P\}$, or $p \in \V{P}$. In either case we let $\bar{C}_p$ be the cone in $\R^2$ formed by the rays $\d^i_\infty$ for $i \in \{1,2\}$: the asymptotes which appear in the construction of $\Diag(P_\fu)$ (or $\Diag(P_\fu,p)$, in the case $P_\fu = P$). We define $C_p$ to be the cone in $\breve{B}$ based at $p$ which is identified with $\bar{C}_p$ by a chart on $\breve{B}$ which identifies -- up to a translation and scale -- $P_\fu$ with $\fu$. We set $W := \bigcup_{p \in J}{C_p}$. To simplify notation we also set $C_\fu := C_{v_\fu}$.

Observe that -- writing $\bar{\d}_E \colon [0,T) \to V$ for the extension of an edge $E$ in $V$ -- there is a ray $\d_E \colon [0,\infty) \to \breve{B}$ such that $\d_E|_{[0,T)} = \bar{\d}_E$. By Lemma~\ref{lem:slope_bound} the restriction of $\d_E$ to $[T,\infty)$ is contained inside $W \subset \breve{B}$ -- the `region dense with rays'. In fact, since Lemma~\ref{lem:V_W_disjoint} shows that the intersection of $V$ and $W$ consists of joints -- and since the tangent space to $W$ at such a vertex $v$ is a strictly convex cone -- we have that $\Ima\d_E \cap V = \Ima\bar{\d}_E$.

\begin{lem}
	\label{lem:V_W_disjoint}
	The intersection $V \cap W$ is equal to the set of vertices of triangles in $\cT$, that is, $V\cap W = J$. In particular the region $W$ does not intersect the interior of any triangle $\fu \in \cT$.
\end{lem}
\begin{proof}
	Fixing an arbitrary vertex $v$ of a triangle in $\cT$, and an arbitrary triangle $\fu \in \cT$, the result follows from the claim that $\fu \cap C_v \subset \{v\}$. Note that the vertex $v$ is either a vertex of $P$, or equal to $v_{\fu'}$ for some $\fu' \in \cT$; if $v \in \V{P}$ we set $\fu' := P$.
	
	We first consider the claim in the case that $\fu$ and $\fu'$ are comparable. Note that -- up to segments of edges of $P$ -- $\fu$ and $\fu'$ are contained in some chart $U_{v_1,v_2}$ of $\breve{B}$. Assuming that $\fu' > \fu$ the cone $C_v$ is contained in the cone $C$ formed by extending the edges of $\fu$ incident to $v_\fu$ (or, if $\fu = P$, incident to a vertex of $P$). Equivalently, replacing $P$ with $P_{\fu}$ in Figure~\ref{fig:better_bound}, the cone $C_v$ is contained in the positive quadrant. Hence we have that $C_v \cap \fu \subset \{v\}$. If instead $\fu' < \fu$, then $\fu \subset R_{\fu''}$ for some successor $\fu''$ of $\fu'$. However, it follows from the definition of $R_{\fu''}$ that $C_{v} \cap R_{\fu''} = \{v\}$.

	Consider the case in which $\fu$ and $\fu'$ are incomparable. Let $x_1$ and $x_2$ be the vertices of $\fu'$ different from $v$. It follows from Proposition~\ref{prop:better_bound} that $C_{v} \setminus R_{\fu'} \subset C_{x_1}\cup C_{x_2}$. Since $\fu'$ is not greater than $\fu$, $\fu \cap R_{\fu'}$ is contained in a single point in the boundary of $R_{\fu'}$ (disjoint from $C_{v}$). Thus it is sufficient to show that $C_{x_i} \cap \fu \subset \{x_i\}$ for each $i \in \{1,2\}$. However, for each $i \in \{1,2\}$, either $C_{x_i} = C_{\fu_i}$ for some $\fu_i < \fu'$, or $x_i \in P < \fu'$. Thus -- iterating this process -- we reduce to the case in which the regions $\fu$ and $\fu'$ are comparable, which we have already shown.
\end{proof}

In Proposition~\ref{prop:no_overlap} we have shown that the triangles in $\cT$ form the chambers of a polyhedral decomposition. We now show that these regions are exactly the chambers defined by the structure $\sS$ described in Construction~\ref{cons:support}. We show that every ray which intersects the interior of the region $V$ is a ray $\d_E$ for an edge $E$ of a triangle in $\cT$.

\begin{lem}
	\label{lem:in_W}
	Fix a vertex $v_i$ of a triangle $\cT$ for each $i \in \{1,2\}$, and a ray $\d_i \subset W$ such that $\d_i \subset C_{v_i}$. Any ray based at the intersection of $\d_1$ and $\d_2$, whose tangent direction is a non-negative linear combination of the direction vectors of $\d_1$ and $\d_2$, is a subset of $W$. 
\end{lem}
\begin{proof}
	Note that $v_i \in R_{E_i}$ for some edge $E_i$ of $P$ for $i \in \{1,2\}$. Let $v$ be a vertex shared by $E_1$ and $E_2$, and note that -- away from $E_1 \cup E_2$ --  $R_{E_1}$ and $R_{E_2}$ are contained in $U_{v,v'}$ for either choice of vertex $v' \neq v$ of $P$. The sets $C_{v_i}\setminus\{v_i\}$ -- and hence $\d_1(0,\infty)$ and $\d_2(0,\infty)$ -- are also contained in the domain of this chart.

	Let $\fu := \fu_1 \wedge \fu_2$ and observe that -- by Proposition~\ref{prop:better_bound} -- if $\fu'$ is a successor to $\fu$ in the partial order on $\cT$ then, for any ray $\d$ whose image is contained in $C_{\fu'}$, there is a $T \in \R_{\geq 0}$ such that $\d((T,\infty)) \subset C_{\fu}$. Note that if $\fu = P$, $C_\fu$ is not defined, but in this case the vertex $v$ is unique and we set $C_\fu := C_v$. By induction there exist $T_1,T_2 \in \R_{\geq 0}$ such that $\d_i((T_i,\infty)) \subset C_\fu$ for each $i \in \{1,2\}$. Since $R_\fu \setminus (\fu \cup C_\fu)$ consists of two connected components, and $v_1$ and $v_2$ are contained in different components, we have that $\d_1\cap\d_2 \subset C_\fu$. However -- as $C_\fu$ is convex -- if $\d_i((T_i,\infty)) \subset C_\fu$ for each $i \in \{1,2\}$ then any positive linear combination of points in $\d_1((T_1,\infty)) \cup \d_2((T_2,\infty))$ is contained in $C_\fu$.
\end{proof}

\begin{prop}
	\label{prop:rays_are_edges}
	The set of rays in $\sS$ with one-dimensional intersection with $V$ is equal to the set of rays $\d_E$ where $E$ ranges over the edges of the triangular regions $\fu \in \Region$.
\end{prop}
\begin{proof}
	Consider an intersection point $v$ between rays $\d_E$ and $\d_{E'}$ in $V$. By Proposition~\ref{prop:no_overlap} there is a triangle $\fu \in \cT$ such that the elements of $\cT$ incident to $v$ are precisely the images of the triangles $S_\fu(\Diag(P_\fu))$ in $\breve{B}$. By Proposition~\ref{prop:rays_and_triangles}, the tangent directions to edges incident to $v$ generate rays of the scattering diagram. Thus we have that $\cT$ is defined by a collection of scattering diagrams in $\breve{B}$. It remains to check that these are the only rays in the compatible structure $\sS$ which intersect the interior of $V$.
	
	By Lemma~\ref{lem:edges_form_rays} and the following discussion we have that rays formed by prolonging edges of chambers in $\cT$ enter a cone $C_\fu \subset W$ for some $\fu \in \cT$ and never re-emerge from $C_\fu$. Note that every ray generated at a vertex $v$ of a triangle in $\cT$ is either is contained in $C_v$, or is a ray $\d_E$ for an edge $E$ of a triangle in $\cT$. Thus any intersection $x$ between rays generated at a vertex $v$ which occurs inside $W$ satisfies the hypotheses of Lemma~\ref{lem:in_W}, and hence all rays generated by scattering diagram at $x$ are contained in a cone $C_v$ for some vertex $v$ of a triangle in $\cT$. Similarly any rays generated at an intersection point of rays generated in $W$ satisfy the conditions of Lemma~\ref{lem:in_W}, and hence all rays in the structure $\sS$ are either of the form $\d_E$, or are disjoint from the interior of $V$.
\end{proof}

The first two points of the statement of Theorem~\ref{thm:superpotential} now follow from Lemma~\ref{lem:edges_form_rays} and Proposition~\ref{prop:rays_are_edges}. Indeed, recall from \cite[Definition~$6.24$]{TropGeom} that the polyhedral $\sP_k$ associated to $\sS[k]$ must satisfy the following:
\begin{enumerate}
	\item Elements of $\sP_k$ must be rational polyhedra with rational vertices.
	\item A sufficiently long initial segment of $\d$ must be a union of edges of $\sP_k$.
	\item Any point of intersection of (non-parallel) rays in $\sS[k-1]$ must be a vertex of $\sP_k$.
\end{enumerate}

Note that the $k$ which appears in this definition is not related to the $k$ which appears in the constructions in \S\ref{sec:diagrams} or \S\ref{sec:gluing_diagrams}. Let $\sP_0$ be the decomposition of $\breve{B}$ shown in Figure~\ref{fig:initial_lines} which decomposes $\breve{B}$ into the region $P$ and three non-compact regions with parallel edges. For any $k$ the set $\sS[k]$ is finite, hence there is a $K(k) \in \Z_{>0}$ such that any intersection point between rays in $\sS[k-1]$ is a vertex $v_\fu$ for some $\fu \in \cT$ such that $d(\fu) < K(k)$. Let $V_k$ -- the subset of $\breve{B}$ appearing in the statement of Theorem~\ref{thm:superpotential} -- be the union of chambers $\fu \in \cT$ such that $d(\fu) \leq K(k)$.

We define $\sP_k$ to be any choice of extension of the decomposition given by $\{\fu \in \cT : d(\fu) \leq K(k)\}$ which meets three conditions specified above. The third point in Theorem~\ref{thm:superpotential} follows from Proposition~\ref{pro:dense_complement}.

\begin{prop}
	\label{pro:dense_complement}
	We have that $\cl(W) \cup V = \breve{B}$. Equivalently, assuming Conjecture~\ref{conj:dense_with_rays}, rays in the structure $\sS$ are dense in the complement of $V$.
\end{prop}
\begin{proof}
	Let $x$ be a point in the complement of $V \cup W$ in $\breve{B}$; we will show that $x \in \cl(W)$. First note that $x \in R_E$ for some edge $E = \conv{v,v'}$ of $P$. We construct a sequence in $\cT$ associated to $x$. Let $\fu_1 \in \cT$ be the triangle, different from $P$, which contains $E$. Note that $x \notin \fu$, and that $R_{\fu_1} \setminus (C_{\fu_1} \cup \fu_1)$ consists of two connected components which are in bijection with the successors of $\fu_1$ in $\cT$. Let $\fu_2$ be the triangle corresponding to the connected component containing $x$. Iterating this process we obtain a monotone sequence $(\fu_i)^\infty_{i=1} \subset \cT$. Let $\ba_i$ denote the Markov triple corresponding to $\fu_i$ for each $i$.
	
	\begin{figure}
		\includegraphics[scale=0.8]{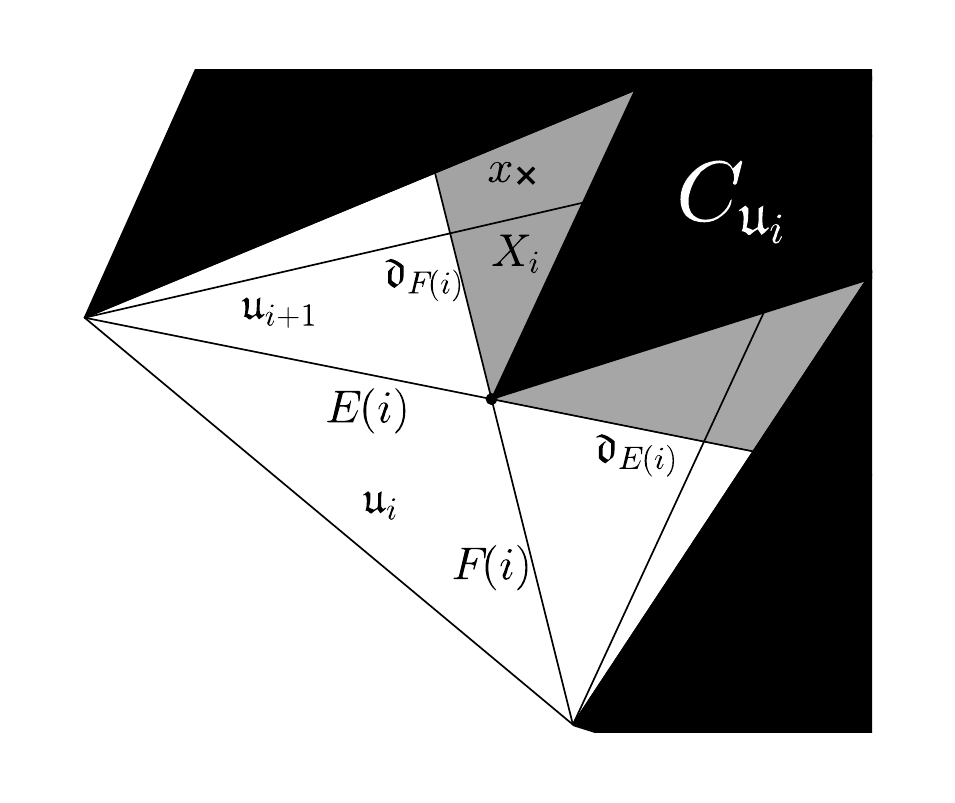}%
		\caption{Regions containing the point $x$.}
		\label{fig:vanishing}
	\end{figure}
	
	For each $i \in \Z_{> 0}$, let $E(i)$ be the edge of $\fu_i$ shared with $\fu_{i+1}$, and let $F(i)$ be the edge of $\fu_i$ containing $v_{\fu_i}$ and different from $E(i)$. Let $e(i)$ and $f(i)$ denote the lengths of these edges, and let $\tilde{E}(i)$ and $\tilde{F}(i)$ denote the corresponding edges of $P_{\fu_i}$. Let $\Gamma$ be the cone formed at $v_{\fu_i}$ by the rays $\d_{E(i)}$ and $\d_{F(i)}$. Note that $x \in \Gamma$, since $R_{\fu_i} \setminus \Gamma \subset V$. Let $X_i$ be the connected component of $(C \cap R_{\fu_i}) \setminus C_{\fu_i}$ which contains $x$, see Figure~\ref{fig:vanishing}. Observe that two edges of $X_i$ are contained in $W$, the remaining edge of $X_i$ extends $F(i)$, and that $X_i \cap \fu_{i+1} = F(i+1)$. It follows from Remark~\ref{rem:small_seg} that the edge of $X_i$ extending $F(i)$ is shorter than $F(i)$. Hence it suffices to show that $f(i) \to 0$ as $i \to 0$.

	First consider the case that the minimal entry in $\ba_i$ is bounded above as $i \to \infty$. In this case the edge $\tilde{F}(i)$ of $P_{\fu_i}$ corresponding to $F(i)$ is eventually constant. Using the bound on $\lambda$ which appears in the proof of Proposition~\ref{prop:better_bound}, we have that $f(i+1) \leq f(i)/2$ for all sufficiently large $i$, hence $f(i)\to 0$. Assume instead that the minimal entry in $\ba_i$ is unbounded, and -- fixing a positive integer $K$ -- choose $i \in \Z_{> 0}$ such that $\tilde{E}(i+1) = \tilde{F}(i)$ and all entries in $\ba_i$ are bounded below by $K$. Observe that -- since $\fu_i$ and $\fu_{i+1}$ are contained the bounded region $R_{\fu_1}$ -- the lengths of their edges have a uniform upper bound $L \in \R$. From the discussion on page~$14$ we have that $e(i+1) \leq f(i)/(3K-1)$. Since one of $e(i+2)$ and $f(i+2)$ is bounded above by $f(i+1)/(3K-1)$, and $E(i+1)$ is an edge of $\fu_{i+2}$, we have that $e(i+2)$ and $f(i+2)$ are bounded above by $x := 2L/(3K-1)$ via the triangle inequality. The sequences $(e(j))^\infty_{j=i+2}$ and $(f(j))^\infty_{j=i+2}$ need not be decreasing, but we observe from repeated application of the triangle inequality that these sequences are bounded by 
	\[
	x\left(1+\frac{1}{3K-1} + \cdots + \frac{1}{(3K-1)^{(j-i-2)}}\right) \leq x\frac{3K-1}{3K-2} = \frac{2L}{3K-2}.
	\]
\end{proof}

This concludes the proof of Theorem~\ref{thm:superpotential}. In fact Theorem~\ref{thm:superpotential_2} is now an immediate consequence of this result and the results of \cite{CPS11}.

\begin{proof}[Proof of Theorem~\ref{thm:superpotential_2}]
	By Proposition~\ref{prop:rays_are_edges} we have that rays in $V$ are unions of edges of triangles $\fu \in \cT$. Moreover, by \cite[Example~$1.15$]{GHKK}, the function attached to each such ray is binomial and -- setting coefficients to be equal to $1$ -- we may assume this function is $1+z^m$. Comparing the formula in \cite{CPS11} for crossing a ray with (algebraic) mutation, we see that the tropical superpotential with basepoint in a triangle $\fu \in \cT$ is precisely the maximally mutable (see \cite{Overarching}) Laurent polynomial with Newton polytope $P_\fu$.
\end{proof}

\bibliographystyle{plain}
\bibliography{tropical_discs}
\end{document}